\documentclass[a4paper,11pt,final]{article}

\usepackage{amsthm,amsfonts,amsmath}
\usepackage{graphicx,color}
\usepackage{hyperref}
\usepackage{enumerate}
\usepackage{geometry}
\usepackage[utf8x]{inputenc}

\newtheorem{proposition}{Proposition}[section]
\newtheorem{theorem}[proposition]{Theorem}
\newtheorem{lemma}[proposition]{Lemma}
\newtheorem{corollary}[proposition]{Corollary}

\newenvironment{proofof}[1]{\smallskip\noindent{\textbf{Proof~of~#1.}}%
  \hspace{1pt}}{\hspace{-5pt}{\nobreak\quad\nobreak\hfill\nobreak%
    $\square$\vspace{2pt}\par}\smallskip\goodbreak}

\numberwithin{equation}{section}
\allowdisplaybreaks[4]

\renewcommand{\phi}{\varphi}
\renewcommand{\theta}{\vartheta}
\renewcommand{\epsilon}{\varepsilon}
\renewcommand{\L}[1]{\mathbf{L^#1}}
\renewcommand{\div}{\mathinner{\mathop{\rm div}}}
\newcommand{\grad}{\mathinner{\mathop{\rm grad}}}

\newcommand{\C}[1]{\mathbf{C^{#1}}}

\newcommand{\modulo}[1]{{\left|#1\right|}}
\newcommand{\norma}[1]{{\left\|#1\right\|}}
\newcommand{\reali}{{\mathbb{R}}}

\newcommand{\tv}{\mathop\mathrm{TV}}
\renewcommand{\O}{\mathinner{\mathcal{O}(1)}}

\newcommand{\spt}{\mathop\mathrm{spt}}
\newcommand{\Id}{\mathop{\mathbf{Id}}}

\newcommand{\caratt}[1]{\chi_{\strut#1}}
\newcommand{\wsto}{\mathinner{\stackrel{\star}{\rightharpoonup}}}

\newcommand{\Caption}[1]{\caption{\narrower{\narrower{\small#1}}}}
\renewcommand{\d}[1]{\mathinner{\mathrm{d}{#1}}}

\begin{document}

\title{A Game Theoretic Approach to\\ Hyperbolic Consensus Problems}

\author{Rinaldo M.~Colombo$^1$ \and Mauro Garavello$^2$}

\footnotetext[1]{INdAM Unit, University of Brescia. \texttt{rinaldo.colombo@unibs.it}}

\footnotetext[2]{Department of Mathematics and its Applications,
  University of Milano Bicocca. \par \texttt{mauro.garavello@unimib.it}}

\maketitle

\begin{abstract}

  \noindent We introduce the use of conservation laws to develop
  strategies in multi-player consensus games. First, basic well
  posedness results provide a reliable analytic setting. Then, a
  general non anticipative strategy is proposed through its rigorous
  analytic definitions and then tested by means of numerical
  integrations.

  \medskip


  \medskip

\end{abstract}

\section{Introduction}
\label{sec:I}

A group of \emph{``leaders''}, or broadcasting agents, aims at getting
the consensus of a variety of individuals. We identify each
individual's opinion with a \emph{``position''} $x$ moving in
$\reali^N$. It is then natural to describe the leaders through their
\emph{''positions''} $P_1, P_2, \ldots, P_k$, also in $\reali^N$. We
are thus lead to the general system of ordinary differential equation
\begin{displaymath}
  \left\{
    \begin{array}{l}
      \dot x
      =
      v\left(t, x, P_1 (t), \ldots, P_k (t)\right)
      \\
      x (0) = \bar x
    \end{array}
  \right.
\end{displaymath}
$t$ being time. The vector field $v$ describes the interaction among
individuals and agents, which can be attractive, repulsive, or a
mixture of the two. Clearly, no linearity assumption can be reasonably
required on $v$, otherwise the interaction between an agent and the
individuals increases as the distance between them increases.

The task of the agent $P_i$, be it attractive or repulsive, is to
maximize its own consensus, i.e., to drive the maximal amount of
individuals (or their opinions) as near as possible to its own target
region $\mathcal{T}_i$ at time $T$, for a suitable non empty region
$\mathcal{T}_i \subset \reali^N$. The time horizon $T$ is finite and
the same for all agents.

A high number of individuals, as well as uncertainties in their
initial positions or specific movements, suggests to describe the
dynamics underneath the present problem through the continuity
equation
\begin{equation}
  \label{eq:25}
  \partial_t \rho
  +
  \div_x \left(\rho \; v\left(t, x, P_1 (t), \ldots, P_k (t)\right)\right)
  =
  0 \,,
\end{equation}
where the description of each individual is substituted by that of the
individuals' density distribution $\rho = \rho (t,x)$, while the goal
of the $i$--th leader is formalized through the minimization of the
quantity
\begin{equation}
  \label{eq:26}
  \mathcal{J}_i = \int_{\reali^N} \rho (T,x) \; d (x, \mathcal{T}_i) \, \d{x}
\end{equation}
where $d (x, \mathcal{T}_i) = \inf_{y \in \mathcal{T}_i} \norma{x-y}$
is the distance between the position $x$ and the target
$\mathcal{T}_i$.

Aim of this paper is to formalize the above setting, to provide basic
well posedness theorems and to initiate the search for
controls/strategies to tackle the above problem. Note that the case
$k=1$ of a single broadcasting agent leads to a control problem, while
the case $k>1$ of $k$ possibly competing agents fits into game theory.

As it is usual in control theory, rather than the agents' positions
$P_i$, it is preferable to use as controls/strategies the agents'
speeds $u_i$, with $u_i \in \reali^N$, subject to a boundedness
constraint of the type $\norma{u_i} \leq U$, for a positive
$U$. Introducing the initial individuals' distribution $\bar \rho$ and
agents' positions $\bar P_1, \ldots, \bar P_k$, the dynamics is then
described by the Cauchy Problem
\begin{equation}
  \label{eq:1}
  \left\{
    \begin{array}{l}
      \partial_t \rho
      +
      \div_x \left(
      \rho \;
      v\!\left(t, x, P_1 (t) , \ldots, P_k (t)\right)
      \right)
      =
      0
      \\
      \rho (0,x) = \bar\rho (x)
    \end{array}
  \right.
  \; \mbox{ where } \;
  \left\{
    \begin{array}{l}
      \dot P_i = u_i (t)
      \\
      P_i (0) = \bar P_i
    \end{array}
  \right.
  \quad i=1, \ldots, k
\end{equation}
where the cost functionals $\mathcal J_i$ are as
in~\eqref{eq:26}. This structure is amenable to the introduction of
several control/game theoretic concepts, from optimal controls to Nash
equilibria, and to the search for their existence. Below we initiate
this study providing the basic analytic framework and tackling the
problem of control/strategies to minimize costs of the
type~\eqref{eq:26}. Various numerical integrations illustrate the
rigorous results obtained.

Note that the present setting, restricted to the case $N=2$, allows
also to describe the individual--continuum interactions considered,
for instance, in~\cite{ColomboMercier}, see
also~\cite{ColomboGaravelloLecureux2012, ColomboHertyMercier},
and~\cite{ColomboPogodaev2012} where an entirely different analytic
framework is exploited. From this point of view, the present work is
related to the vast literature on crowd and swarm dynamics, see the
recent works~\cite{MR3110059, MR3285315, MR3431278, 4118472,
  PiccoliPouradierScharf, MR3432849, MR3285316} or the
review~\cite{BellomoDogbe} and the references therein.

Concerning our choice of the conservation law~\eqref{eq:25}, we stress
that typical of equations of this kind is the finite speed both of
propagation of information and of the support of the density. This is
in contrast with the typical situation in standard differential games
ruled by parabolic equations.

In the next section we first provide the basic notation and
definitions, then we provide basic well posedness results and
introduce a reasonable non anticipative strategy. Section~\ref{sec:NI}
is devoted to sample applications, while all analytic proofs are
deferred to Section~\ref{sec:TD}.

\section{Analytic Results}
\label{sec:AR}

Throughout, the positive time $T$ and the maximal speed $U$ are
fixed. For $a,b \in \reali$, denote
$\langle a, b \rangle = [\min\{a, b\}, \max \{a, b\}]$.  By
$\mathcal{L}^N$ we mean the Lebesgue measure in $\reali^N$. The open,
respectively closed, ball in $\reali^m$ centered at $u$ with radius
$U$ is $B_{\reali^m}(u,U)$, respectively
$\overline{B_{\reali^m}(u,U)}$; when the space is clear, we shorten to
$B(u,U)$ or $\overline{B(u,U)}$. In $\reali$, $\modulo{\,\cdot\,}$ is
the absolute value, while $\norma{\,\cdot\,}$ is the Euclidean norm in
$\reali^N$. The norm in the functional space $\mathcal{F}$ is denoted
$\norma{\,\cdot\,}_{\mathcal{F}}$. The space $\C0 (A; \reali^n)$ of
the $\reali^n$-valued functions defined on the subset $A$ of
$\reali^m$ is equipped with the norm
$\norma{f}_{\C0 (A;\reali^n)} = \sup_{x \in A} \norma{f (x)}$.
Throughout, $\tv (\,\cdot\,)$ stands for the total variation,
see~\cite[Chapter~5]{EvansGariepy}. For a measurable function $\rho$
defined on $\reali^N$, $\spt \rho$ is its support,
see~\cite[Proposition~4.17]{brezis}.

Introduce $P \equiv (P_1, \ldots, P_k)$, so that $P \in \reali^m$ with
$m = k\,N$, and rewrite~\eqref{eq:1} as
\begin{equation}
  \label{eq:4}
  \left\{
    \begin{array}{l}
      \partial_t \rho
      +
      \div_x \left(
      \rho \;
      v\!\left(t, x, P (t)\right)
      \right)
      =
      0
      \\
      \rho (0,x) = \bar\rho (x)
    \end{array}
  \right.
  \quad \mbox{ where } \quad
  \left\{
    \begin{array}{l}
      \dot P = u (t)
      \\
      P (0) = \bar P \,.
    \end{array}
  \right.
\end{equation}
Below, recurrent assumptions on the function $v$ in~\eqref{eq:4} are
the following:

\begin{description}
\item[\emph{(v0)}:] The vector field
  $v \in \C0 ([0,T] \times \reali^N \times \reali^m; \reali^N)$ is
  such that for all $t \in [0,T]$ and $P \in \reali^m$, the map
  $x \to v (t, x, P)$ is in $\C{0,1} (\reali^N; \reali^N)$.
\item[\emph{(v1)}:] \textbf{\emph{(v0)}} holds and moreover
  \vspace{-0.5\baselineskip}
  \begin{itemize}
  \item for all $t \in [0,T]$ and $P \in \reali^m$, the map
    $x \to v (t, x, P)$ is in $\C{1,1} (\reali^N; \reali^N)$;
  \item for all $t \in [0,T]$ and $x \in \reali^N$, the map
    $P \to v (t, x, P)$ is in $\C{0,1} (\reali^m; \reali^N)$.
  \end{itemize}
\end{description}

\noindent We now prove well posedness and basic estimates
for~\eqref{eq:1} or, equivalently, (\ref{eq:4}).

\begin{proposition}
  \label{prop:0}
  Fix positive $T$ and $U$. Let $v$ satisfy~\textbf{(v0)}. For any
  $\bar\rho \in \L1 (\reali^N; \reali)$, $\bar P \in \reali^m$ and
  $u \in \L\infty ([0,T]; \overline{B_{\reali^m}(0,U)})$,
  problem~\eqref{eq:4} admits the unique solution
  \begin{displaymath}
    \rho (t,x)
    =
    \bar\rho\left(X (0;t,x)\right)
    \exp \left(
      -\int_0^t \div_x v\left(\tau, X (\tau; t,x), P (\tau)\right) \d\tau
    \right)
  \end{displaymath}
  where $t \to X (t; \bar t, \bar x)$ solves $\left\{
    \begin{array}{l@{}}
      \dot x = v\left(t,x,P (t)\right)
      \\
      x (\bar t) = \bar x
    \end{array}
  \right.$
  and $\displaystyle P (t) = \bar P + \int_0^t u (\tau)\d\tau$ for
  $t \in [0,T]$. Moreover, if $v$ satisfies~\textbf{(v1)} and
  $u_1,u_2 \in \L\infty ([0,T]; \overline{B_{\reali^m} (0,U)})$, then
  (with obvious notation) for all $t \in [0,T]$,
  \begin{eqnarray}
    \label{eq:29}
    \norma{X_1 (t; 0, \bar x) - X_2 (t; 0, \bar x)}
      & \leq
      & C \, t \, e^{C\, t} \; \norma{P_1 - P_2}_{\C0 ([0,t]; \reali^m)}
    \\
    \label{eq:34}
    \norma{\rho_1 (t) - \rho_2 (t)}_{\L1 (\reali^N; \reali)}
      & \leq
      & C \, \Bigl(
        \norma{\grad_x \bar\rho}_{\L\infty (\reali^N; \reali^N)} \;
        \mathcal{L}^N \! \left(B(\spt\bar\rho, C t e^{C t})\right)
    \\
    \nonumber
      &
      & \qquad\qquad
        +
        \norma{\bar\rho}_{\L1 (\reali^N; \reali)} \,
        (1 + C\, t)
        \Bigr)
        t \, e^{2Ct} \, \norma{P_1-P_2}_{\C0 ([0,t]; \reali^m)}
  \end{eqnarray}
  where $C$ is independent of the initial datum, more precisely:
  \begin{equation}
    \label{eq:21}
    C
    =
    \max
    \left\{
      \begin{array}{ll}
        \norma{v}_{\L\infty ([0,t]\times\reali^N \times \reali^m; \reali^N)} ,
        & \norma{D_x v}_{\L\infty ([0,t]\times\reali^N\times\reali^m; \reali^{N\times N})} ,
        \\
        \norma{D_P v}_{\L\infty ([0,t]\times\reali^N\times\reali^m; \reali^{N\times m})} ,
        & \norma{\grad_x \div_x v}_{\L\infty ([0,t]\times\reali^N\times\reali^m; \reali^N)}
      \end{array}
    \right\} \,.
  \end{equation}
\end{proposition}

\noindent The proof is deferred to Section~\ref{sec:TD}. Here, the
term \emph{``solution''} means Kru\v zkov
solution~\cite[Definition~1]{kruzkov}, which is also a strong solution
as soon as $\bar\rho$ is smooth. A straightforward consequence of the
above Lemma is the following convergence result, which we state
without proof.

\begin{corollary}
  \label{cor:weak}
  Fix positive $T$ and $U$. Let $v$ be bounded and
  satisfy~\textbf{(v1)}, $\bar\rho \in \L1 (\reali^N; \reali)$ and
  $\bar P \in \reali^m$. If
  $u_n, u_* \in \L\infty ([0,T]; \overline{B_{\reali^m} (0,U)})$ are
  such that $u_n \wsto u_*$ in $\L\infty ([0,T]; \reali^m)$ as
  $n \to +\infty$, then, up to a subsequence,
  \begin{displaymath}
    P_n
    \to
    P_*
    \mbox{ in } \C0 ([0,T]; \reali^N)
    \quad \mbox{ and } \quad
    \begin{array}{r@{\;}c@{\;}l@{\mbox{ in }}l}
      \rho_n (t)
      & \to
      &  \rho_* (t)
      & \L1 (\reali^N; \reali) \mbox{ for all }t \in [0,T]\,,
      \\
      \rho_n
      & \to
      & \rho_*
      & \C0\left([0,T]; \L1 (\reali^N; \reali)\right) \,.
    \end{array}
  \end{displaymath}
  If $\bar\rho \in \C1 (\reali^N;\reali)$, then
  $\begin{array}{rcll} \grad_x \rho_n (t) & \to & \grad_x \rho_* (t) &
    \mbox{ in } \L1 (\reali^N; \reali) \mbox{ for all }t \in [0,T]
     \\
     \grad_x \rho_n & \to & \grad_x \rho_* & \mbox{ in }
                                             \C0\left([0,T]; \L1
                                             (\reali^N; \reali)\right)
   \end{array}$.
 \end{corollary}

 The $i$-th leader $P_i$ seeks a control
 $u_i \in \L\infty \big([0,T]; \overline{B (0,U)}\big)$ that minimizes
 the cost
 $\mathcal{J}_i = \int_{\reali^N} \rho (T,x) \, \psi (x) \, \d{x}$,
 which reduces to~\eqref{eq:26} in the case
 $\psi (x) = d (x, \mathcal{T}_i)$. Assume first that $P_i$ knows in
 advance the strategies $u_j$, for $j \neq i$, of the other
 controllers $P_j$, so that its task amounts to
 minimize~\eqref{eq:5}. Corollary~\ref{cor:weak} ensures that
 \begin{equation}
   \label{eq:5}
   \begin{array}{ccccc}
     \mathcal{J}_i
     & \colon
     & \L\infty\left([0,T]; \overline{B (0,U)}\right)
     & \to
     & \reali
     \\
     &
     & u_i
     & \to
     & \displaystyle\int_{\reali^N} \rho (T,x) \, \psi (x) \, \d{x}
   \end{array}
 \end{equation}
 is weak$\star$ continuous. Hence, by the weak$\star$ compactness of
 $\L\infty([0,T]; \overline{B (0,U)})$, there exists an optimal
 control $u_i^*$ that minimizes $\mathcal{J}_i$.

 Note however that this approach can hardly be used in a game
 theoretic setting, since it requires that $P_i$ is aware of all other
 strategies $u_j$, $j \neq i$, on the whole time interval $[0,T]$,
 which is unreasonable whenever different agents are competing.

 \smallskip

 We now proceed towards the definition of a non anticipative
 strategy. To this aim, we simplify the notation setting $P = P_i$,
 $u = u_i$, $\mathcal{J} = \mathcal{J}_i$ and comprising within the
 time dependence of the function $v$ all the other strategies $u_j$,
 for $j \neq i$. In this setting, we define a \emph{non anticipative}
 strategy $u$ for the controller $P$, i.e., a strategy $u = u (t)$
 that depends only on $\rho$ at times $s \in \left[0, t\right[$.

 For a positive (suitably small) $\Delta t$, we seek the best choice
 of a speed $w \in \overline{B(0,U)}$ on the interval
 $[t, t+\Delta t]$ such that the solution $\rho_w = \rho_w (\tau,x)$
 to
 \begin{equation}
   \label{eq:7}
   \left\{
     \begin{array}{l}
       \partial_\tau \rho_w
       +
       \div_x \left(
       \rho_w \;
       v\!\left(t, x, P (t) + (\tau-t)w\right)
       \right)
       =
       0
       \\
       \rho_w (t,x) = \rho (t, x)
     \end{array}
   \right.
   \qquad \tau \in [t, t+\Delta t]
 \end{equation}
 is likely to best contribute to decrease the value of
 $\mathcal{J}$. Remark that the dependence of $v$ on $t$
 in~\eqref{eq:7} is \emph{frozen} at time $t$. It is this choice that
 will later lead to a non anticipative strategy.

 We now verify that~\eqref{eq:7} is well posed.

 \begin{lemma}
   \label{lem:1}
   Fix positive $T$, $U$, and $\Delta t \in ]0, T[$. Let
   $v \in \C{0,1} ([0,T] \times \reali^N \times \reali^N; \reali^N)$.
   For any $\bar\rho \in \L1 (\reali^N; \reali)$,
   $\bar P \in \reali^N$, $u \in \L\infty ([0,T]; \overline{B(0,U)})$,
   $t \in \left[0, T - \Delta t\right[$ and $w \in \reali^N$,
   problem~\eqref{eq:7} admits a unique solution given by
   \begin{equation}
     \label{eq:28}
     \rho_w (\tau, x)
     =
     \rho \left(t, X_{t,w} (t; \tau, x)\right) \;
     \exp \left(
       -\int_t^\tau
       \div_x v\left(t, X_{t,w}(s; \tau, x), P (t) + (s-t) \, w\right) \d{s}
     \right)
   \end{equation}
   where
   \begin{equation}
     \label{eq:22}
     \tau \to X_{t,w} (\tau;\bar t,x)
     \quad \mbox{ solves } \quad
     \left\{
       \begin{array}{l}
         \xi' = v \left(t, \xi, P (t) + (\tau-t)w\right)
         \\
         \xi (\bar t) = x
       \end{array}
     \right.
     \mbox{ for } \bar t, \tau \in [t, t+\Delta t] \,.
   \end{equation}
   Moreover, if $\bar\rho \in \C{1,1} (\reali^N; \reali)$ and
   $\mathcal{L}^N (\spt \bar\rho) < +\infty$, for all
   $w_1,w_2 \in \reali^N$
   \begin{equation}
     \label{eq:3}
     \begin{array}{rcl}
       &
       & \norma{\rho_{w_1} (t+\Delta t) - \rho_{w_2} (t+\Delta t)}_{\L1 (\reali^N; \reali)}
       \\
       & \leq
       & \left(
         \norma{\grad_x \bar\rho}_{\L\infty (\reali^N; \reali^N)} \;
         \mathcal{L}^N \left(
         \spt \bar\rho, C e^{C \Delta t}
         \Delta t
         \right)
         +
         (1+C\,\Delta t) \norma{\bar\rho}_{\L1 (\reali^N; \reali)}
         \right)
       \\
       &
       & \quad
         \times
         C e^{2C \Delta t} \, (\Delta t)^2 \, \norma{w_1 - w_2}
     \end{array}
   \end{equation}
   where
   \begin{equation}
     \label{eq:14}
     C = \max \left\{
       \begin{array}{l}
         \norma{v}_{\L\infty ([t, t+\Delta t] \times \reali^N \times B (P (t), \Delta t \, U);\reali^N)},
         \\
         \norma{D_x v}_{\L\infty ([t, t+\Delta t] \times \reali^N \times B (P (t), \Delta t \, U); \reali^{N\times N})},
         \\
         \frac12 \;
         \norma{D_P v}_{\L\infty ([t, t+\Delta t] \times \reali^N \times \reali^N; \reali^{N\times N})},
         \\
         \norma{D_P \div_x v}_{\L\infty ([t, t+\Delta t] \times \reali^N \times \reali^N; \reali^N)}
       \end{array}
     \right\}
   \end{equation}
 \end{lemma}

 \noindent (Above and in the sequel, $\xi' = \frac{\d\xi}{\d\tau}$).
 The proof of Lemma~\ref{lem:1} is deferred to Section~\ref{sec:TD}.

 \smallskip

 In the case of the functional~\eqref{eq:26}, a natural choice for the
 agent $P$ at time $t$ is then to choose a speed $w$ on the time
 interval $[t, t+\Delta t]$ to minimize the quantity
 \begin{equation}
   \label{eq:8}
   \mathcal{J}_{t, \Delta t} (w)
   =
   \int_{\reali^N} \rho_w (t+\Delta t,x) \; \psi (x) \, \d{x} \,.
 \end{equation}

 \begin{proposition}
   \label{prop:1}
   Fix positive $T, U$, $\Delta t \in ]0, T[$, and fix a boundedly
   supported initial datum $\bar \rho \in \L1 (\reali^N; \reali)$,
   $\bar P \in \reali^N$, a speed law
   $v \in \C{0,1} ([0,T] \times \reali^N \times \reali^N; \reali^N)$
   and a weight $\psi \in \L\infty (\reali^N; \reali)$.  Then, with
   the notation in~\eqref{eq:4} and~\eqref{eq:7}, for any
   $t \in \left[0, T\right[$ and $\Delta t \in \left]0, T-t\right]$
   the map
   \begin{displaymath}
     \begin{array}{ccccc}
       \mathcal{J}_{t, \Delta t}
       & \colon
       & \reali^N
       & \to
       & \reali
       \\
       &
       & w
       & \to
       & \displaystyle
         \int_{\reali^N} \rho_w (t+\Delta t,x) \; \psi (x) \d{x}
     \end{array}
   \end{displaymath}
   is well defined and Lipschitz continuous.
 \end{proposition}

 \noindent The main theorem now follows, providing explicit
 information on a non anticipative optimal choice of $w$.

 \begin{theorem}
   \label{thm:bo}
   Fix positive $T$, $U$, and $\Delta t \in ]0, T[$. Let
   $v \in \C2 ([0,T]\times\reali^N\times\reali^N; \reali^N)$,
   $\psi \in \L\infty (\reali^N; \reali)$ and a boundedly supported
   $\bar\rho \in \C1 (\reali^N; \reali)$. Define $\rho$ as the
   solution to~\eqref{eq:4} and $\rho_w$ as the solution
   to~\eqref{eq:7}, for a $w \in \reali^N$.  The map
   \begin{equation}
     \label{eq:33}
     \begin{array}{ccccc}
       \mathcal{J}_{t, \Delta t}
       & \colon
       & \reali^N
       & \to
       & \reali
       \\
       &
       & w
       & \to
       & \displaystyle \int_{\reali^N} \rho_w (t+\Delta t,x) \; \psi (x) \d{x}
     \end{array}
   \end{equation}
   admits the expansion
   \begin{equation}
     \label{eq:18}
     \mathcal{J}_{t, \Delta t} (w+\delta_w)
     =
     \mathcal{J}_{t, \Delta t} (w)
     +
     \grad_w\mathcal{J}_{t, \Delta t} (w)  \cdot \delta_w + o (\delta_w)
     \qquad \mbox{ as } w \to 0
   \end{equation}
   where, as $\Delta t \to 0$,
   \begin{equation}
     \label{eq:24}
     \begin{array}{@{}r@{\;}c@{\;}l@{}}
       \!\!\!\!\!\!
       \grad_w\mathcal{J}_{t, \Delta t} (w)\!
       & = \!
       & \displaystyle
         \dfrac{(\Delta t)^2}{2} \!\!\!
         \int_{\reali^N} \!\!
         \Big[ \!
         \grad_x \rho (t,x)
         D_P v\! \left(t, x, P (t)\right)
         \!\!-\!\!
         \rho (t,x) \grad_P \div_x v \! \left(t, x, P (t)\right)
         \!\! \Big] \!
         \psi (x) \d{x}
       \\[20pt]
       &
       & + o (\Delta t)^2 .
     \end{array}
     \!\!\!\!\!\!\!\!\!\!\!\!\!\!\!\!\!\!\!\!\!\!
   \end{equation}
 \end{theorem}

 \noindent The proof is deferred to Section~\ref{sec:TD}. On the basis
 of Theorem~\ref{thm:bo}, the definition of an effective non
 anticipative strategy for $P_i$ can be easily achieved as
 follows. Split the interval $[0,T]$ in smaller portions
 $\left[ t_\ell , t_{\ell+1} \right[$, where
 $t_\ell = \ell \, \Delta t$. On each of them, define
 $u_i (t) = w_\ell$, where $w_\ell$ minimizes on $\overline{B (0,U)}$
 the cost $\mathcal{J}_{t_\ell, \Delta t}$ defined
 in~\eqref{eq:33}. The leading term in the right hand side
 of~\eqref{eq:24} is independent of $w$, so that for $\Delta t$ small
 it is reasonable to choose
 \begin{displaymath}
   w_\ell = -
   \dfrac{ U \displaystyle\int_{\reali^N}
     \Big[
     \grad_x \rho (t_\ell,x) \;
     D_P v \left(t_\ell, x, P_i (t_\ell)\right)
     -
     \rho (t_\ell,x) \; \grad_P div_x v \left(t_\ell, x, P_i (t_\ell)\right)
     \Big]
     \psi (x) \d{x}}{\norma{\displaystyle\int_{\reali^N}
       \Big[
       \grad_x \rho (t_\ell,x) \;
       D_P v\! \left(t_\ell, x, P_i (t_\ell)\right)
       -
       \rho (t_\ell,x) \; \grad_P \div_x v \! \left(t_\ell, x, P_i (t_\ell)\right)
       \Big]
       \psi (x) \d{x}}}
 \end{displaymath}
 as long as the denominator above does not vanish, in which case we
 set $w_\ell=0$. Remark that, through the term $\rho_\ell$, the right
 hand side above depends on all the past values
 $w_0, \ldots, w_{\ell-1}$ attained by $u_i$. Formally, in the limit
 $\Delta t \to 0$, the above relations thus leads to a delayed
 integrodifferential equation.

 \section{Examples}
 \label{sec:NI}

 This section presents a few numerical integrations of the
 game~\eqref{eq:1}--\eqref{eq:26} in which a strategy is chosen as
 described in Section~\ref{sec:AR}.

 As the function $v$ in~\eqref{eq:1}, we choose
 \begin{equation}
   \label{eq:velocity-simulation}
   v (t, x, P)
   =
   \sum_{i = 1}^k a_i \left(\norma{x - P_i}\right) \, \left(P_i - x\right)\,,
 \end{equation}
 where $P \equiv (P_1, \ldots, P_k)$ and
 $a_i\colon \reali^+ \to \reali$, $i \in \left\{1, \cdots, k\right\}$,
 is chosen so that~\textbf{\emph{(v1)}} holds.  In other words, at
 time $t$, the velocity $v (t, x, P)$ of the individual at $x$ is the
 sum of $k$ vectors, each of them parallel to the straight line
 through $x$ and the agent's position $P_i$ and its strength depends
 on the distance between $x$ and $P_i$.  Typically, the functions
 $a_i$ is chosen so that for all $t$ and $P$, the map
 $x \to v (t, x, P)$ is either compactly supported, or vanishes as
 $\norma{x} \to +\infty$.  Note that $a_i > 0$ whenever $P_i$ is
 attractive, while $a_i < 0$ in the repulsive case. In the examples
 below, the targets are single points and, correspondingly, the cost
 $\psi_i$ is the distance from that point.

 With reference to~\eqref{eq:4}, in each of the integrations below we
 use the Lax--Friedrichs algorithm~\cite[Section~4.6]{LeVequeBook2002}
 with dimensional splitting~\cite[Section~19.5]{LeVequeBook2002} to
 integrate the conservation law, while the usual explicit forward
 Euler method is adequate for the ordinary differential equation. To
 ease the presentations of the results, we fix the space dimension
 $N = 2$.  Correspondingly, in each of the rectangular domains
 $\Omega$ considered below, we fix a rectangular regular grid
 consisting of $n_x \times n_y$ points. The treatment of the boundary
 $\partial\Omega$ is eased whenever the vector $v$ along
 $\partial\Omega$ points inward.

 \subsection{A Single Agent}
 \label{sse:1-agent-non-ant}

 Consider~(\ref{eq:4}) in the numerical domain
 $\Omega = [0, 10] \times [0, 10]$, with
 \begin{equation}
   \label{eq:piff-come-back}
   \begin{array}{@{}r@{\;}c@{\;}l@{}}
     N
     & =
     & 2 \,,
     \\
     k
     & =
     & 1 \,,
     \\
     m
     & =
     & 2 \,,
   \end{array}
   \qquad
   \begin{array}{r@{\;}c@{\;}l@{}}
     a_1 (\xi)
     & =
     & \frac{1}{\xi} \, e^{-\xi/10}\,,
     \\
     v (t,x,P)
     & =
     & e^{-\norma{x-P_1}/10} \, (P_1-x) \,,
     \\
     U
     & =
     & 3/2
   \end{array}
   \qquad
   \begin{array}{r@{\;}c@{\;}l@{}}
     \bar \rho
     & =
     & \caratt{[6,8]\times[2,8]},
     \\
     \bar P_1
     & \equiv
     & (3,2),
   \end{array}
   \qquad
   \mathcal{T}_1
   = \left\{(1,8)\right\} .
 \end{equation}
 We now compute the solution to~\eqref{eq:1} with $u$ piecewise
 constant given by the strategy~\eqref{eq:24}, constant on intervals
 $[j\, \Delta t, (j+1) \, \Delta t]$, where $\Delta t = 1/100$. The
 resulting solution, obtained on a grid of $6000 \times 6000$ cells,
 is displayed in Figure~\ref{fig:SingleAgent}.
 \begin{figure}[htpb]
   \centering
   \includegraphics[width=0.245\linewidth,trim=75 30 30
   10]{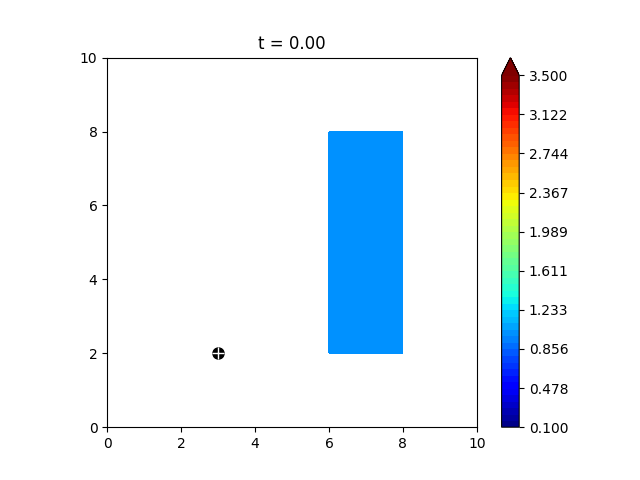}%
   \includegraphics[width=0.245\linewidth,trim=75 30 30
   10]{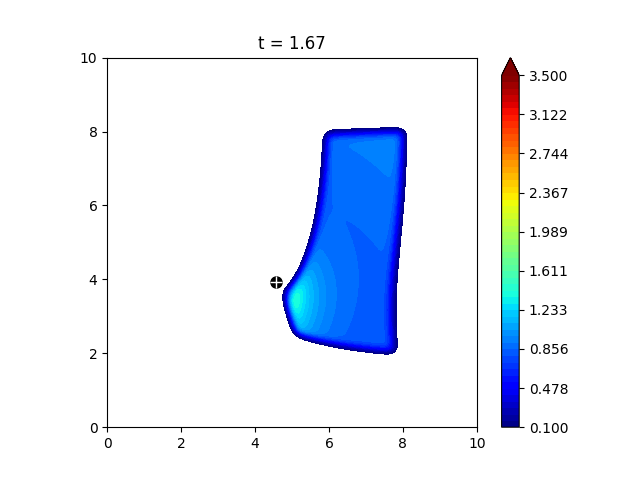}%
   \includegraphics[width=0.245\linewidth,trim=75 30 30
   10]{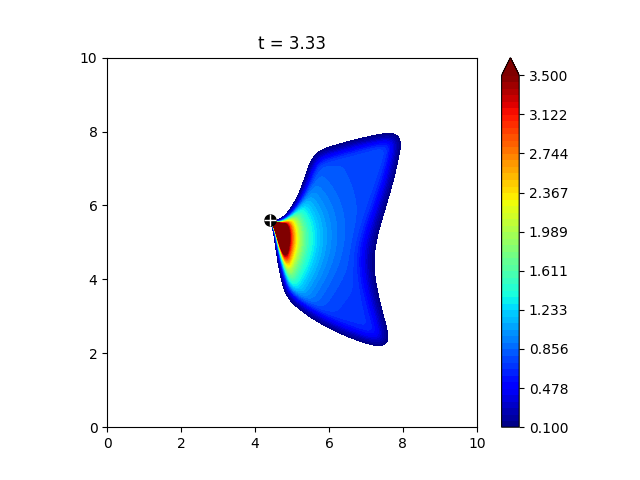}%
   \includegraphics[width=0.245\linewidth,trim=75 30 30
   10]{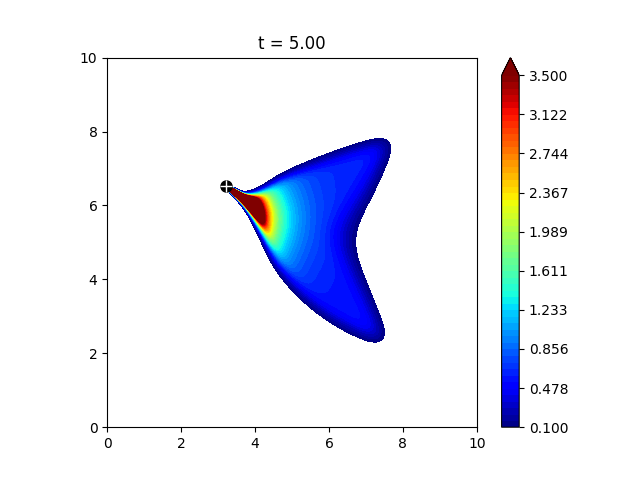}\\
   \includegraphics[width=0.245\linewidth,trim=75 30 30
   10]{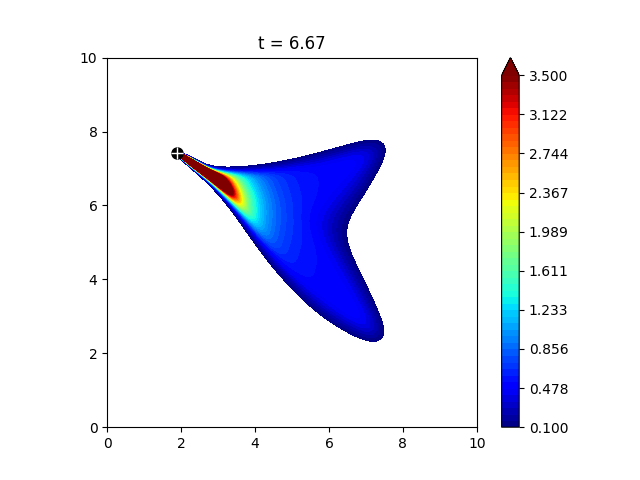}%
   \includegraphics[width=0.245\linewidth,trim=75 30 30
   10]{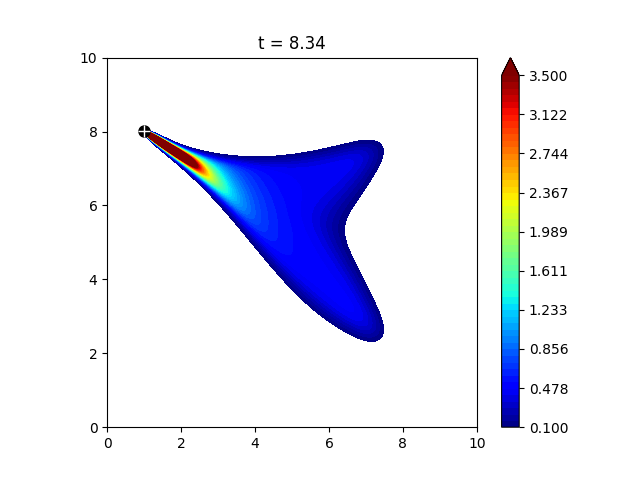}%
   \includegraphics[width=0.245\linewidth,trim=75 30 30
   10]{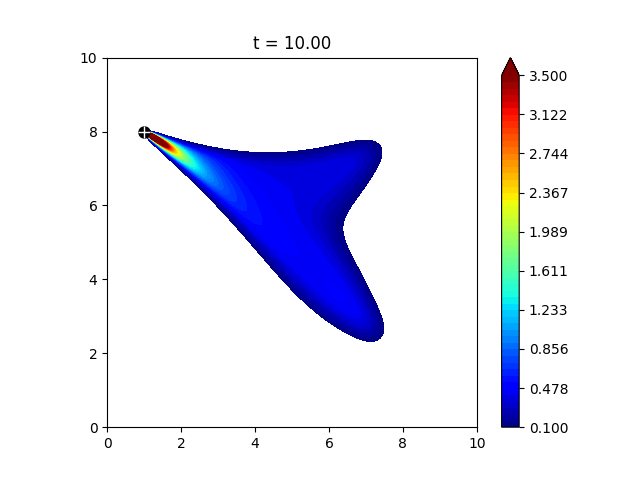}%
   \includegraphics[width=0.25\linewidth]{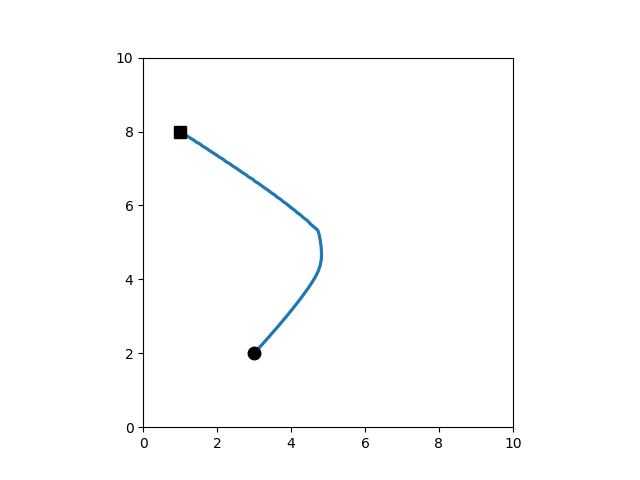}\\
   \Caption{Numerical integration of~\eqref{eq:1} with the
     strategy~\eqref{eq:24} and the
     parameters~\eqref{eq:piff-come-back}. The first 7 figures depict
     the contour plots of the solution $\rho$ and the position of $P$,
     the bottom right diagram displays the trajectory of $P_1$, whose
     initial position is $\left(3,2\right)$, drawn as a black circle.
     Here, $\Delta t = 1/100$.  Note that, in spite of the myopic
     nature of the strategy~\eqref{eq:24}, the leader first moves to
     the right and then turns back to the left.}
   \label{fig:SingleAgent}
 \end{figure}
 Remarkably, although the strategy~\eqref{eq:24} is fully
 \emph{myopic}, the leader $P_1$ does not move directly towards the
 target $\mathcal{T}_1$. On the contrary, it first moves to the right
 to \emph{collect} a higher quantity of individuals and then moves
 back to the left; see Figure~\ref{fig:SingleAgent}. The resulting
 cost~\eqref{eq:26} is $29.33$. %$29.329$.

 \subsection{Two Competing Attractive Agents}

 % directory: paper\_1

 We now test the strategy~\eqref{eq:24} against an \emph{a priori}
 assigned strategy. More precisely, we let
 $\Omega = [0, 10] \times [0, 10]$, with
 \begin{equation}
   \label{eq:23}
   \begin{array}{@{}r@{\;}c@{\;}l@{}}
     N
     & =
     & 2 \,,
     \\
     k
     & =
     & 2 \,,
     \\
     m
     & =
     & 4 \,,
c   \end{array}
   \qquad
   \begin{array}{r@{\;}c@{\;}l@{}}
     a_1 (\xi)
     & =
     & \frac{1}{\xi} \, e^{-\xi/5}\,,
     \\
     a_2 (\xi)
     & =
     & \frac{1}{\xi} \, e^{-\xi/5}\,,
     \\
     v (t,x,P)
     & =
     & \mbox{as in~\eqref{eq:velocity-simulation}}
       \,,
     \\
     U
     & =
     & 3/2\,,
   \end{array}
   \qquad
   \begin{array}{r@{\;}c@{\;}l@{}}
     \bar \rho
     & =
     & \caratt{[7,9]\times[3,7]},
     \\
     \bar P_1
     & =
     & (8,5),
     \\
     \bar P_2
     & =
     & (8,5),
   \end{array}
   \qquad
   \begin{array}{@{}r@{\;}c@{\;}l@{}}
     \mathcal{T}_1
     & =
     &  \left\{(1, 9)\right\},
     \\
     \mathcal{T}_2
     & =
     &  \left\{(1, 1)\right\}.
   \end{array}
 \end{equation}
 Moreover, we first assign to $P_1$ the rectilinear trajectory
 \begin{equation}
   \label{eq:20}
   P_1 (t)
   =
   \left[
     \begin{array}{@{}c@{}}
       8\\5
     \end{array}
   \right] + \left[
     \begin{array}{@{}c@{}}
       -7/10\\ 2/5
     \end{array}
   \right] t\,,
   \quad \mbox{ corresponding to } \quad
   u_1 (t) = \left[
     \begin{array}{@{}c@{}}
       -7/10\\ 2/5
     \end{array}
   \right] \,.
 \end{equation}

 The agent $P_1$ follows a rectilinear trajectory towards the target
 located at the point $\left(1, 9\right)$. At the final time $T = 10$,
 the cost of player $P_1$, when alone, is $11.73$, see
 Table~\ref{tab:paper1}.
 % Il costo è 11.729
 \begin{figure}[htpb]
   \centering
   \includegraphics[width=0.245\linewidth,trim=75 30 30
   10]{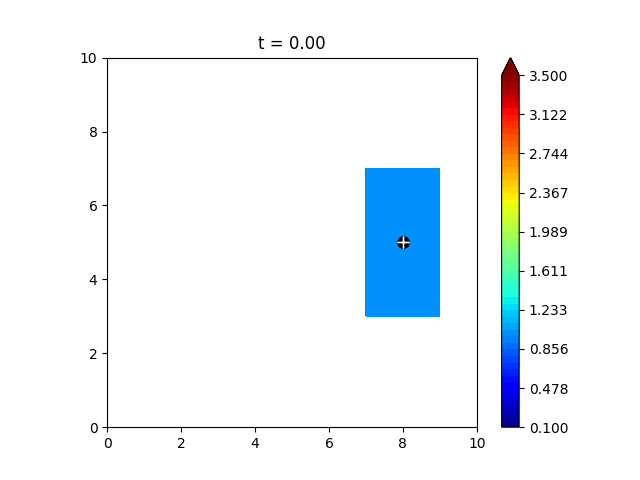}%
   \includegraphics[width=0.245\linewidth,trim=75 30 30
   10]{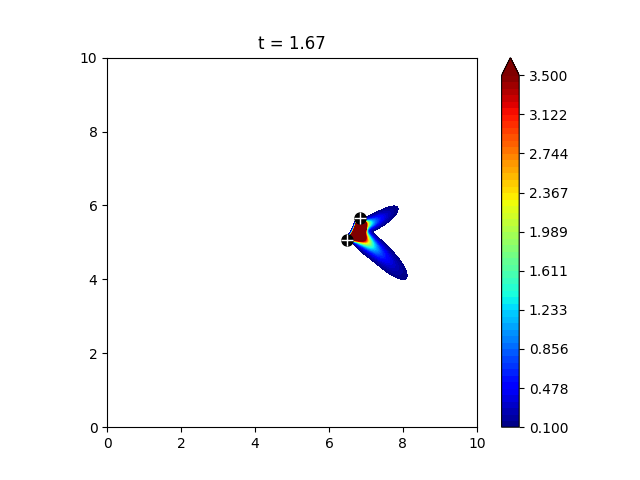}%
   \includegraphics[width=0.245\linewidth,trim=75 30 30
   10]{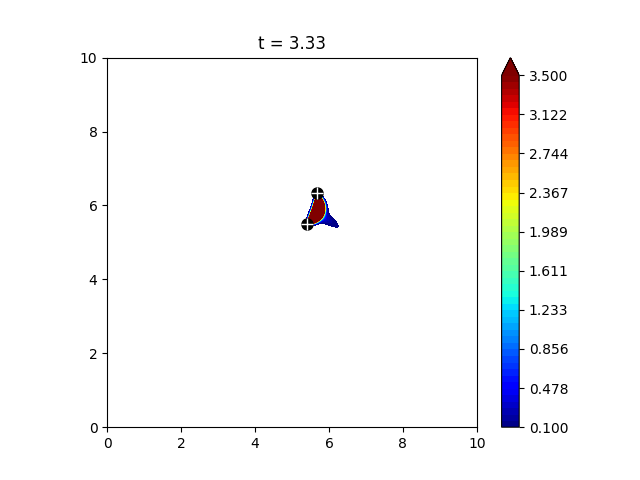}%
   \includegraphics[width=0.245\linewidth,trim=75 30 30
   10]{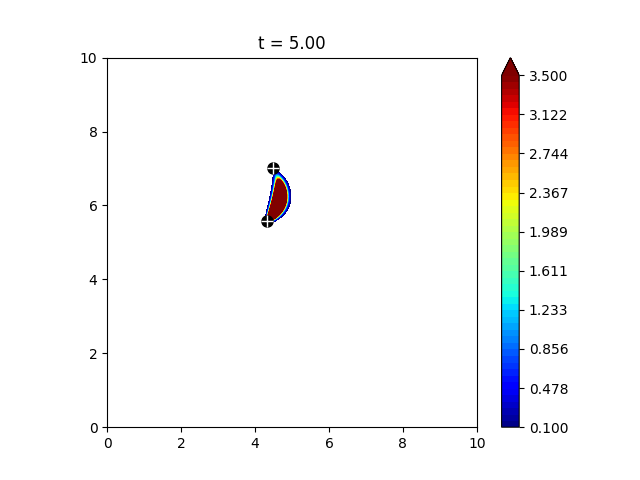}\\
   \includegraphics[width=0.245\linewidth,trim=75 30 30
   10]{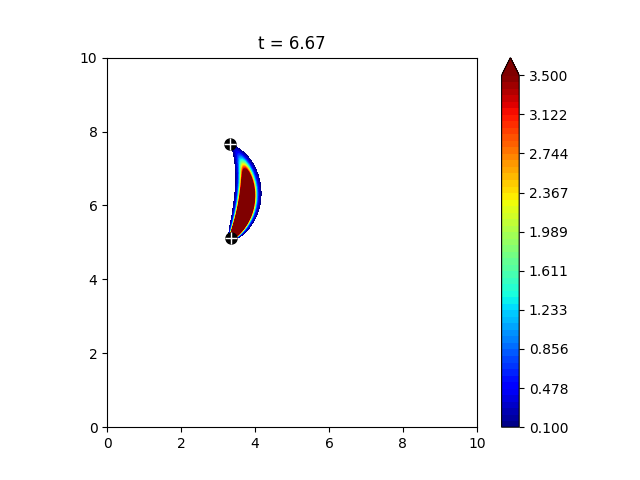}%
   \includegraphics[width=0.245\linewidth,trim=75 30 30
   10]{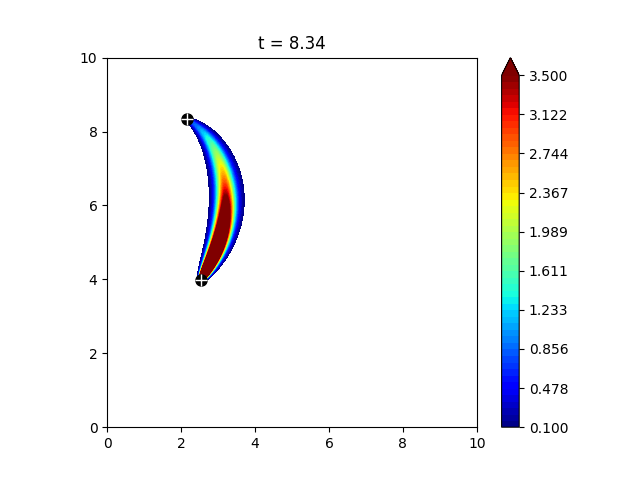}%
   \includegraphics[width=0.245\linewidth,trim=75 30 30
   10]{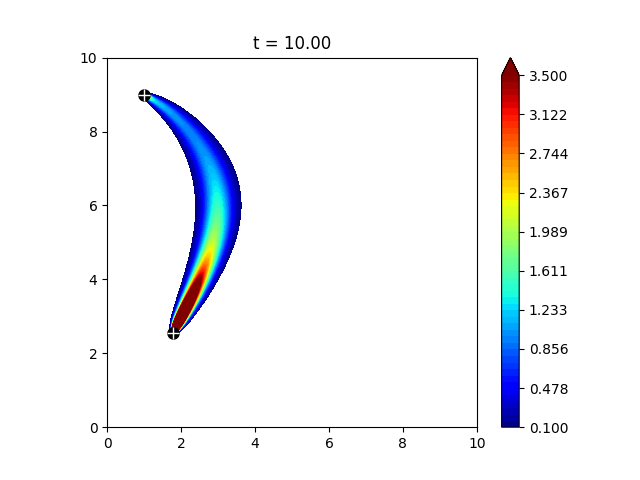}%
   \includegraphics[width=0.25\linewidth]{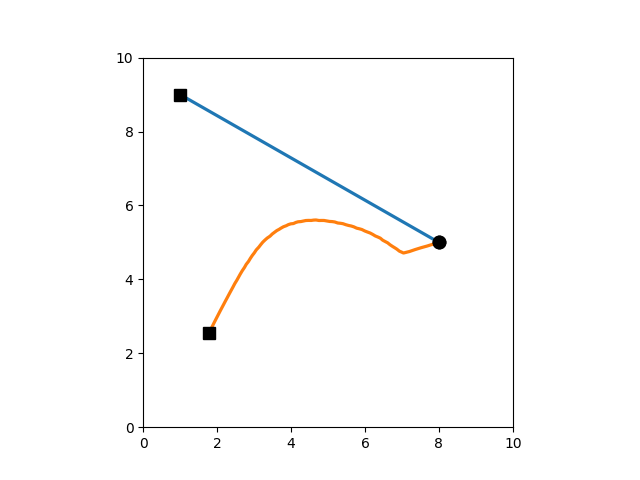}\\
   \Caption{Numerical integration of~\eqref{eq:1}--\eqref{eq:23} with
     two players. $P_1$ is assigned strategy~\eqref{eq:20}, while
     $P_2$ uses~\eqref{eq:24} with $\Delta t = 1/100$. The first 7
     figures depict the contour plots of the solution $\rho$, the
     bottom right diagram displays the trajectories of $P_1$ and
     $P_2$, whose initial positions are as in~\eqref{eq:23}. $P_2$
     wins.}
   \label{fig:paper1}
 \end{figure}
 Then, we insert also the player $P_2$, assigning its strategy $u_2$
 by means of~\eqref{eq:24}. The result is shown in
 Figure~\ref{fig:paper1}: strategy~\eqref{eq:24} leads to the victory
 of $P_2$. Here, $P_2$ first moves slightly up, superimposing its
 attraction to that of $P_1$. Then, it bends downwards attracting more
 individuals than $P_1$; see Figure~\ref{fig:paper1}.  The agent $P_2$
 goes initially towards the target located at $\left(1, 1\right)$,
 but, after a small amount of time, it turns up, attracting more
 individuals than $P_1$.

 The results pertaining the costs $\mathcal{J}_1$ and $\mathcal{J}_2$
 are summarized in Table~\ref{tab:paper1}.
 \begin{table}[htpb]
   \centering
   \begin{tabular}{ccccc}
     Strategy of $P_1$
     & Strategy of $P_2$
     & Cost $\mathcal{J}_1$ of $P_1$
     & Cost $\mathcal{J}_2$ of $P_2$
     \\
     \eqref{eq:20}
     & (absent)
     & $11.73$ % $11.730$
     & //
     \\
     \eqref{eq:20}
     & \eqref{eq:24}
     & $36.41$ % $36.414$
     & $32.65$ % $32.654$
     \\
     \eqref{eq:24}
     & \eqref{eq:24}
     & $33.42$ % $33.418$
     & $33.42$ % $33.425$
   \end{tabular}
   \Caption{Values of the costs $\mathcal{J}_1$ and $\mathcal{J}_2$ resulting from~\eqref{eq:1}--\eqref{eq:23} with different strategies. On the first line, $P_1$ plays alone. The second line shows that strategy~\eqref{eq:24} wins against~\eqref{eq:20}. The third line correctly shows that, in a symmetric situation, if both players use strategy~\eqref{eq:24} the result is even.}
   \label{tab:paper1}
 \end{table}
 Note the sharp increase in the cost $\mathcal{J}_1$ due to $P_2$
 entering the game. The last line confirms that if the two players
 have the same effect on the individuals, the initial configuration is
 symmetric and both players use strategy~\eqref{eq:24}, then the
 players break even.

 % Costi in based-sample: 36.41360969 e 32.65429204 Costi in
 % only-feedback: 33.41782936 e 33.42532815 Costo per il singolo:
 % 11.72965066

 \subsection{Automatic Cooperation among Repulsive Agents}

 % directory repulsion6

 The strategy introduced in Section~\ref{sec:AR} fosters a sort of
 \emph{cooperation} among agents having the same
 goal. Consider~\eqref{eq:4} with cost~\eqref{eq:26} and parameters,
 where $i=1, \ldots, 6$,
 \begin{equation}
   \label{eq:35}
   \begin{array}{@{}r@{\;}c@{\;}l@{}}
     N
     & =
     & 2 \,,
     \\
     k
     & =
     & 6 \,,
     \\
     m
     & =
     & 12 \,,
     \\
     T
     & =
     & 5 \,,
   \end{array}
   \qquad
   \begin{array}{r@{\;}c@{\;}l@{}}
     a_i (\xi)
     & =
     & - \frac{1}{\xi} \, e^{-\xi/5}\,,
     \\
     v (t,x,P)
     & =
     & \mbox{as in~\eqref{eq:velocity-simulation}}
       \,,
     \\
     U
     & =
     & 1\,,
   \end{array}
   \qquad
   \begin{array}{r@{\;}c@{\;}l@{}}
     \bar \rho
     & =
     & \caratt{[6,8]\times[3,7]},
     \\
     \bar P_1
     & =
     & (1,2),
     \\
     \bar P_2
     & =
     & (1,4),
     \\
     \bar P_3
     & =
     & (1,6),
   \end{array}
   \qquad
   \begin{array}{r@{\;}c@{\;}l@{}}
     \bar P_4
     & =
     & (1,8),
     \\
     \bar P_5
     & =
     & (9,4),
     \\
     \bar P_6
     & =
     & (9,6),
     \\
     \mathcal{T}_i
     & =
     &  \left\{(5, 5)\right\}.
   \end{array}
 \end{equation}
 Then, the application of the strategy defined in Section~\ref{sec:AR}
 automatically results in a team play, see
 Figure~\ref{fig:repulsion6}.
 \begin{figure}[htpb]
   \centering
   \includegraphics[width=0.245\linewidth,trim=75 30 30
   10]{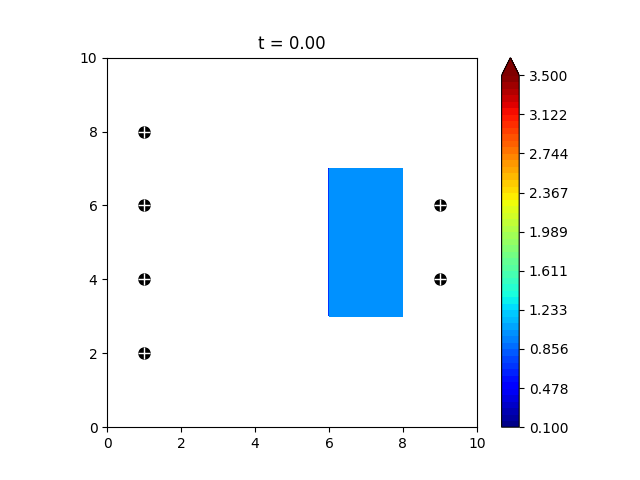}%
   \includegraphics[width=0.245\linewidth,trim=75 30 30
   10]{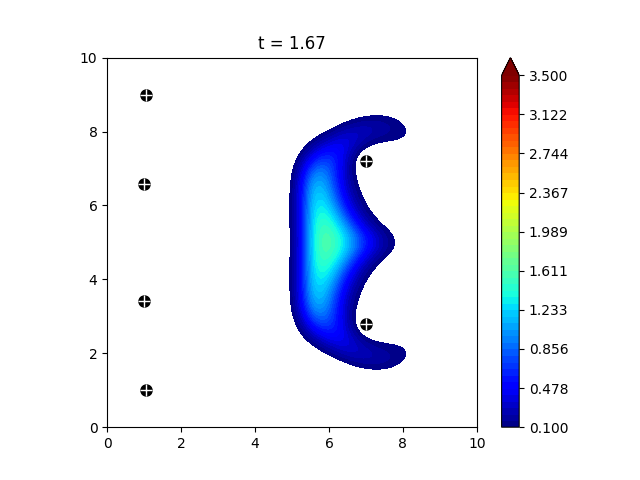}%
   \includegraphics[width=0.245\linewidth,trim=75 30 30
   10]{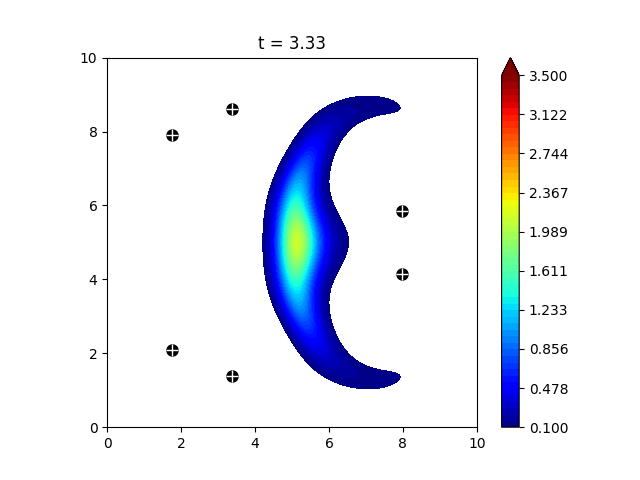}%
   \includegraphics[width=0.245\linewidth,trim=75 30 30
   10]{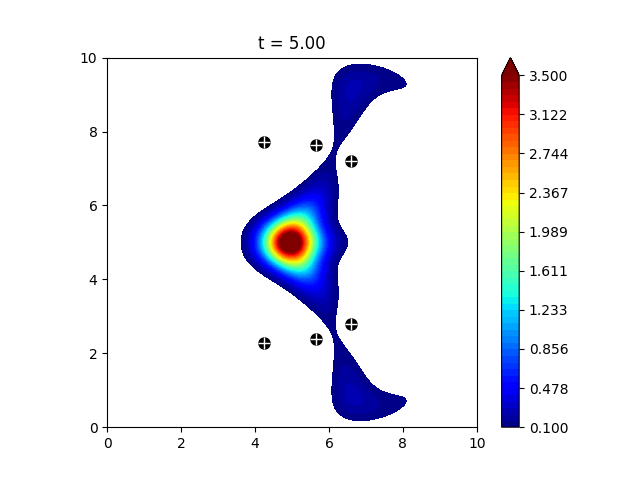}%
   \Caption{Integration of~\eqref{eq:4} with
     parameters~\eqref{eq:35}. The $6$ players are assigned the same
     target and automatically cooperate.  After time $T=5$, a portion
     of the individuals escapes the numerical domain, distorting the
     computation of the cost.}
   \label{fig:repulsion6}
 \end{figure}
 This integration is computed through a grid $3000 \times 3000$. The
 resulting final cost, common to all players, is $10.54$.

 % Il costo al tempo 5.00087905 è 10.5451381

 \subsection{Competition/Cooperation among Attractive/Repulsive
   Agents}

 % directories attr-rep_1_2_C e attr-rep_1_2_NC

 Finally, the following integrations of~\eqref{eq:4} show first that
 cooperation arises also between attractive and repulsive
 agents. Then, it emphasizes the clear difference between cooperation
 and competition. Consider first the case
 \begin{equation}
   \label{eq:36}
   \begin{array}{@{}r@{\;}c@{\;}l@{}}
     N
     & =
     & 2 \,,
     \\
     k
     & =
     & 3 \,,
     \\
     m
     & =
     & 6 \,,
     \\
     T
     & =
     & 5 \,,
   \end{array}
   \qquad
   \begin{array}{r@{\;}c@{\;}l@{}}
     a_1 (\xi) = a_3 (\xi)
     & =
     & - \frac{1}{\xi} \, e^{-\xi/5}\,,
     \\
     a_2 (\xi)
     & =
     & \frac{1}{\xi} \, e^{-\xi/5}\,,
     \\
     v (t,x,P)
     & =
     & \mbox{as in~\eqref{eq:velocity-simulation}}
       \,,
     \\
     U
     & =
     & 1\,,
   \end{array}
   \qquad
   \begin{array}{r@{\;}c@{\;}l@{}}
     \bar \rho
     & =
     & \caratt{[1,2]\times[3,7]},
     \\
     \bar P_1
     & =
     & (1,1),
     \\
     \bar P_2
     & =
     & (1,5),
     \\
     \bar P_3
     & =
     & (1,9),
   \end{array}
   \qquad
   \begin{array}{r@{\;}c@{\;}l@{}}
     \mathcal{T}_1
     & =
     &  \left\{(9, 5)\right\}.
   \end{array}
 \end{equation}
 whose solution is depicted in Figure~\ref{fig:attr-rep}, first line.
 \begin{figure}[htpb]
   \centering
   \includegraphics[width=0.245\linewidth,trim=75 30 30
   10]{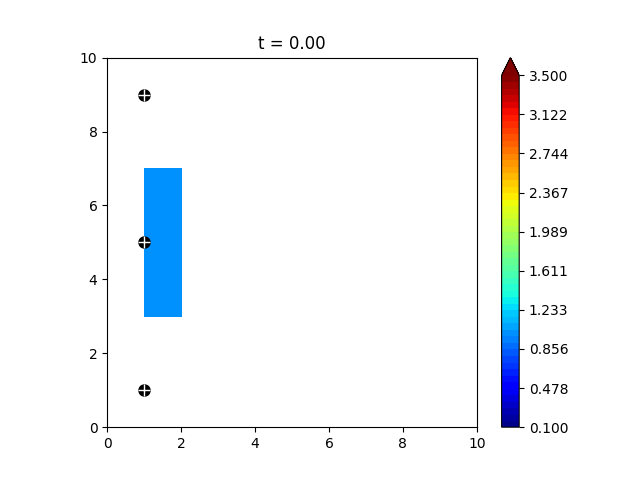}%
   \includegraphics[width=0.245\linewidth,trim=75 30 30
   10]{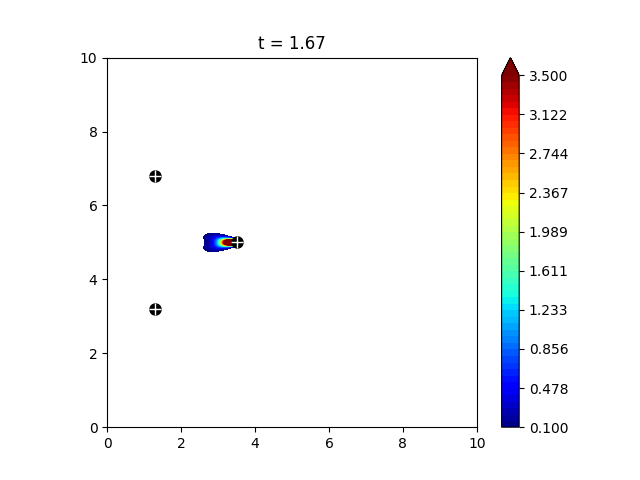}%
   \includegraphics[width=0.245\linewidth,trim=75 30 30
   10]{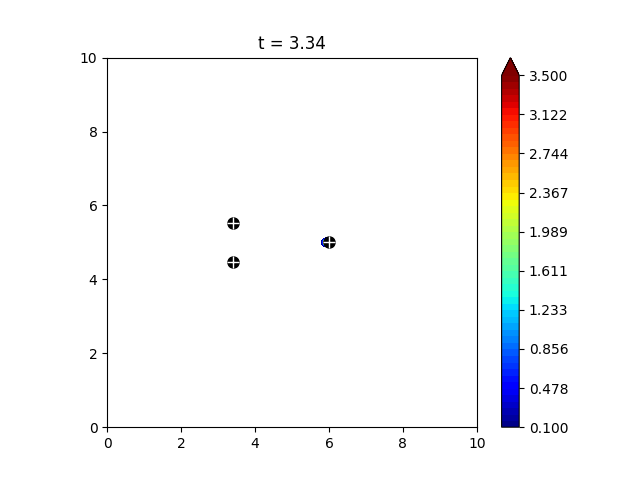}%
   \includegraphics[width=0.245\linewidth,trim=75 30 30
   10]{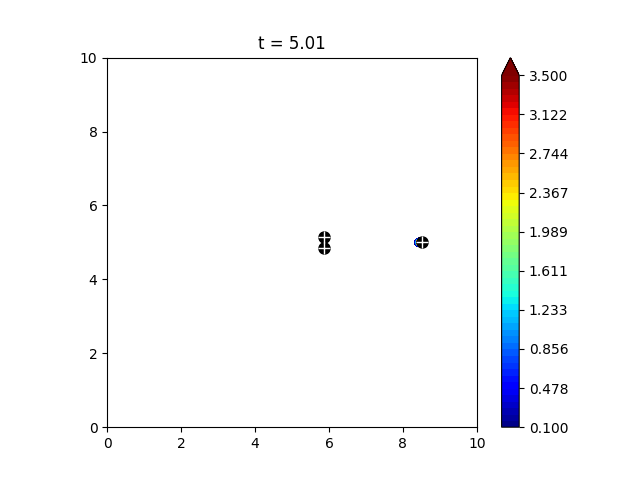}\\
   \includegraphics[width=0.245\linewidth,trim=75 30 30
   10]{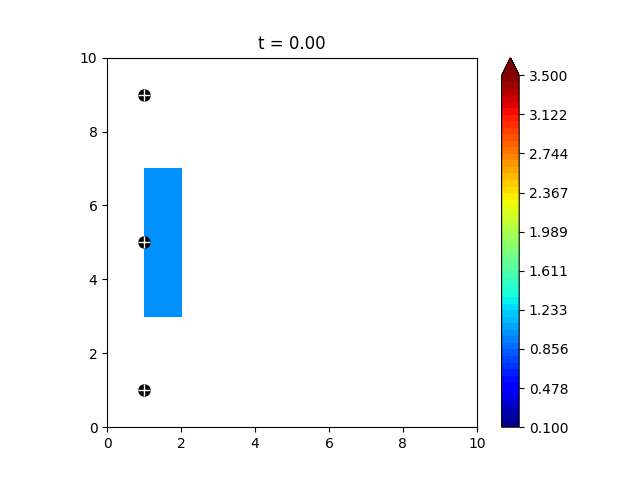}%
   \includegraphics[width=0.245\linewidth,trim=75 30 30
   10]{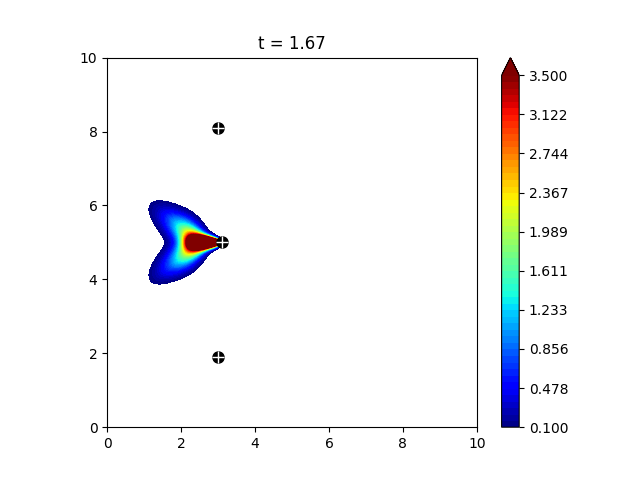}%
   \includegraphics[width=0.245\linewidth,trim=75 30 30
   10]{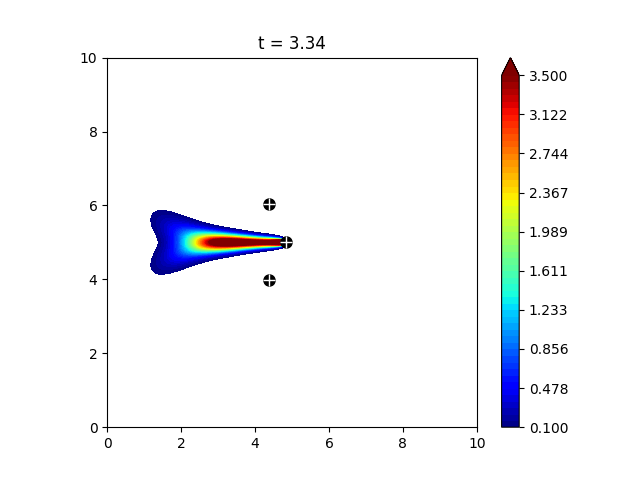}%
   \includegraphics[width=0.245\linewidth,trim=75 30 30
   10]{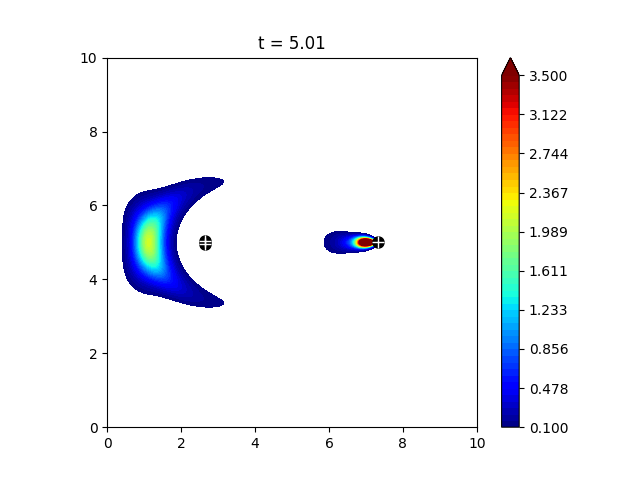}%
   \Caption{Upper line, integration of~\eqref{eq:4} with
     parameters~\eqref{eq:36} and with the same cost for all players
     $\psi_1 (x) = \psi_2 (x) = \psi_3 (x) = d (x, \mathcal{T}_1)$. In
     the lower line, we set $\psi_1 = \psi_3 = -\psi_2$ as
     in~\eqref{eq:37}. As a result, $P_1$ and $P_3$ \emph{steal} most
     of the followers to $P_2$. In both cases, $P_1$ and $P_3$
     are repulsive, while $P_2$ is attracting.}
   \label{fig:attr-rep}
 \end{figure}
 % Il costo al tempo 5.00763147,   2.03673403,
 The final cost is $2.04$, the density $\rho$
 being highly concentrated near to the target $\mathcal{T}_1$. Then,
 we keep the same parameters, but modify the costs of $P_1$ and $P_3$
 setting
 \begin{equation}
   \label{eq:37}
   \psi_1 (x) = \psi_3 (x) = - d (x, \mathcal{T}_1)
   \quad \mbox{ and } \quad
   \psi_2 (x) = d (x, \mathcal{T}_1) \,.
 \end{equation}
 The resulting evolution is in Figure~\ref{fig:attr-rep}, second
 line. Note that $P_1$ and $P_3$ follow now a quite different
 trajectory, \emph{``cutting''} the density $\rho$ so that the final
 cost of $P_2$ raises to $26.68$.  In both integrations, the mesh
 consists of $3000 \times 3000$ points.
 % Il costo al tempo 5.00763147,  26.68479828

 \section{Technical Details}
 \label{sec:TD}

 Throughout, the continuous dependence of $V$ and $v$ on $t$ can be
 easily relaxed to mere measurability. In view of the applications
 below, the following result on ordinary differential equations
 deserves being recalled.

 \begin{lemma}[{\cite[Chapter~3]{BressanPiccoliBook}}]
   \label{lem:ode}
   Let $V_1, V_2 \in \C{0} ([0,T] \times \reali^N; \reali^N)$ be such
   that the maps $x \to V_i (t,x)$ are in
   $\C{0,1} (\reali^N; \reali^N)$ for $i=1,2$ and for all
   $t \in [0,T]$. Then, for all
   $(\bar t, \bar x) \in [0,T] \times \reali^N$ and
   $i \in \left\{1, 2\right\}$, the Cauchy Problem
   \begin{equation}
     \label{eq:10}
     \left\{
       \begin{array}{l}
         \dot x = V_i (t,x)
         \\
         x (\bar t) = \bar x
       \end{array}
     \right.
   \end{equation}
   admits, on the interval $[0,T]$, the unique solution
   $t \to X_i (t; \bar t, \bar x)$ and the following estimate holds,
   for all $t \in [0,T]$:
   \begin{equation}
     \label{eq:12}
     \begin{split}
       \norma{X_1 (t; \bar t, \bar x) - X_2 (t; \bar t, \bar x)} \le
       {} & \norma{V_1-V_2}_{\L1 (\langle \bar t, t\rangle; \L\infty
         (\reali^N; \reali))}
       \\
       & \times \exp \left( \norma{D_x V_2}_{\L\infty(\langle \bar t,
           t\rangle \times\reali^N; \reali^{N\times N})}
         \modulo{t-\bar t} \right) \,.
     \end{split}
   \end{equation}
   If moreover $x \to V_i (t, x) \in \C1 (\reali^N; \reali^N)$ for all
   $t \in [0,T]$, the map $x \to X_i (t; \bar t, \bar x)$ is
   differentiable and its derivative
   $t \to D_x X_i (t; \bar t, \bar x)$ solves the linear matrix
   ordinary differential equation
   \begin{equation}
     \label{eq:2}
     \left\{
       \begin{array}{l}
         \dot Y = D_x V_i\left(t, X_i (t; \bar t, \bar x)\right) \, Y
         \\
         Y (\bar t) = \Id \,.
       \end{array}
     \right.
   \end{equation}
 \end{lemma}

 \begin{lemma}
   \label{lem:CL1}
   Let $V \in \C{0} ([0,T] \times \reali^N; \reali^N)$ be such that
   the map $x \to V (t,x)$ is in $\C{1} (\reali^N; \reali^N)$ for
   $i = 1, 2$ and for all $t \in [0,T]$.  Then, for all
   $\bar t \in \left[0,T\right[$, $i \in \left\{1,2\right\}$, and
   $\bar \rho \in \L1 (\reali^N; \reali)$, the Cauchy Problem
   \begin{equation}
     \label{eq:6}
     \left\{
       \begin{array}{l}
         \partial_t \rho + \div_x \left(\rho \, V (t,x)\right) = 0
         \\
         \rho (\bar t, x) = \bar \rho (x)
       \end{array}
     \right.
   \end{equation}
   admits, on the interval $[\bar t,T]$, the unique {Kru\v zkov}
   solution
   \begin{equation}
     \label{eq:9}
     \rho (t,x)
     =
     \bar\rho\left(X (\bar t; t, x)\right) \;
     \exp \left(
       -\int_{\bar t }^t \div_x V\left(\tau, X (\tau;t,x)\right) \d\tau
     \right)
   \end{equation}
   and if $\spt\bar\rho$ is bounded, then
   \begin{equation}
     \label{eq:11}
     \spt \rho (t)
     \subseteq
     B
     \left(
       \spt\rho (\bar t),
       \norma{V}_{\L\infty (\langle\bar t, t\rangle \times \spt \bar\rho; \reali^N)}
       \modulo{t-\bar t}\,
       e^{\norma{D_x V}_{\L\infty (\langle \bar t, t \rangle \times \reali^N; \reali^{N\times N})} \modulo{t-\bar t}}
     \right) .
   \end{equation}
 \end{lemma}

\begin{proof}
  The fact that~\eqref{eq:9} solves~\eqref{eq:6} in {Kru\v zkov} sense
  follows from~\cite[Lemma~5.1]{ColomboHertyMercier}. To
  prove~\eqref{eq:11}, compute
  \begin{eqnarray*}
    \norma{X (t; \bar t, x) - x}
    & \leq
    & \modulo{\int_{\bar t}^t
      \norma{V \left(\tau; X (\tau; \bar t, x)\right)} \d\tau}
    \\
    & \leq
    & \modulo{
      \int_{\bar t}^t
      \left(
      \norma{V (\tau, x)}
      +
      \norma{V \left(\tau; X (\tau; \bar t, x)\right) - V (\tau,x)}
      \right) \d\tau
      }
    \\
    & \leq
    &
      \norma{V }_{\L\infty (\langle\bar t, t\rangle \times \spt \bar\rho; \reali^N)}
      \modulo{t-\bar t}
    \\
    &
    &\quad
      +
      \modulo{
      \int_{\bar t}^t
      \norma{D_x V}_{\L\infty (\langle\bar t, t \rangle \times \reali^N; \reali^{N\times N})}
      \norma{X (\tau;\bar t, x) - x} \d\tau
      }
  \end{eqnarray*}
  and by Gr\"onwall Lemma, see, e.g., \cite[Chapter~3,
  Lemma~3.1]{BressanPiccoliBook},
  \begin{displaymath}
    \norma{X (t; \bar t, x) - x}
    \leq
    \norma{V }_{\L\infty (\langle\bar t, t\rangle \times \spt \bar\rho; \reali^N)} \,
    \modulo{t-\bar t} \;
    e^{\norma{D_x V}_{\L\infty (\langle\bar t, t \rangle \times \reali^N; \reali^{N\times N})}
      \modulo{t-\bar t}} \,,
  \end{displaymath}
  completing the proof.
\end{proof}

\begin{lemma}
  \label{lem:CL2}
  Let $V_1, V_2 \in \C{0} ([0,T] \times \reali^N; \reali^N)$ be such
  that both maps $x \to V_i (t,x)$, $i = 1, 2$, are in
  $\C{1,1} ([0,T]\times\reali^N; \reali^N)$. If
  $\bar\rho \in \C{0,1} (\reali^N;\reali)$, then
  \begin{eqnarray*}
    \!\!\!\!\!\!
    &
    & \norma{\rho_1 (t)-\rho_2 (t)}_{\L1 (\reali^N; \reali)}
    \\
    \!\!\!\!\!\!
    & \leq
    & \norma{\grad_x \bar\rho}_{\L\infty (\reali^N; \reali^N)} \;\,
      \mathcal{L}^N \! \left(
      \spt \bar\rho, C e^{C \modulo{t-\bar t}}
      \modulo{t-\bar t}
      \right)
      e^{2C \modulo{t-\bar t}}
      \norma{V_1-V_2}_{\L1 (\langle \bar t, t\rangle; \L\infty (\reali^N; \reali))}
    \\
    \!\!\!\!\!\!
    &
    & \!\!
      +
      \left(
      \norma{\div_x (V_1 - V_2)}_{\L\infty (\langle\bar t, t\rangle \times \reali^N; \reali)}
      +
      C \norma{V_1-V_2}_{\L1 (\langle \bar t, t\rangle; \L\infty (\reali^N; \reali))}
      \right)
      % \\
      % &
      % & \qquad \times
          \norma{\bar\rho}_{\L1 (\reali^N; \reali)}
          e^{2C \modulo{t-\bar t}}
          \modulo{t-\bar t}
  \end{eqnarray*}
  where
  \begin{equation}
    \label{eq:13}
    C
    =
    \max_{i=1,2} \left\{
      \begin{array}{l}
        \norma{V_i}_{\L\infty (\langle\bar t, t\rangle \times \reali^N; \reali^N)}
        \\
        \norma{D_x V_i}_{\L\infty (\langle\bar t, t\rangle \times \reali^N; \reali^{N\times N})}
        \\
        \norma{\grad_x \div_x V_i}_{\L\infty (\langle\bar t, t\rangle \times \reali^N; \reali^N)}
      \end{array}
    \right\} \,.
  \end{equation}
\end{lemma}

\begin{proof} Using~\eqref{eq:9} and the triangle inequality, we have
  \begin{displaymath}
    \norma{\rho_1 (t) - \rho_2(t)}_{\L1 (\reali^N; \reali)}
    \leq
    (I) + (II) + (III)
  \end{displaymath}
  where
  \begin{eqnarray*}
    (I)
    &  =
    & \int_{\reali^N}
      \modulo{\bar\rho \left(X_1 (\bar t; t, x)\right) - \bar\rho \left(X_2 (\bar t; t, x)\right)}
      \exp\modulo{\int_{\bar t}^t \div_x V_1 \left(\tau, X_1 (\tau; t, x)\right) \d\tau}
      \d{x}
    \\
    (II)
    & =
    & \int_{\reali^N} \bar\rho\left(X_2 (\bar t; t, x)\right)
    \\
    & \times
    & \!\!\!
      \modulo{
      \exp\left[-\int_{\bar t}^t \div_x V_1\left(\tau, X_1 (\tau; t, x)\right) \d\tau\right]
      -
      \exp\left[-\int_{\bar t}^t \div_x V_2\left(\tau, X_1 (\tau; t, x)\right) \d\tau\right]
      }
      \d{x}
    \\
    (III)
    & =
    &  \int_{\reali^N} \bar\rho\left(X_2 (\bar t; t, x)\right)
    \\
    & \times
    & \!\!\!
      \modulo{
      \exp\left[-\int_{\bar t}^t \div_x V_2\left(\tau, X_1 (\tau; t, x)\right) \d\tau\right]
      -
      \exp\left[-\int_{\bar t}^t \div_x V_2\left(\tau, X_2 (\tau; t, x)\right) \d\tau\right]
      }
      \d{x}
  \end{eqnarray*}
  and we now bound the three terms separately. To estimate $(I)$,
  observe that by~\eqref{eq:11}
  \begin{displaymath}
    \bigcup_{i=1}^2 \spt \rho_i (t)
    \subseteq
    B\left(\spt \bar\rho, \max_{i=1,2}
      \norma{V_i}_{\L\infty (\langle\bar t, t\rangle \times \spt\bar\rho; \reali^N)}
      \exp \left(
        \norma{D_x V_i}_{\L\infty (\langle \bar t, t \rangle \times \reali^N; \reali^N)}
        \modulo{t-\bar t}
      \right)
      \modulo{t-\bar t}
    \right)
  \end{displaymath}
  and, using~\eqref{eq:12},
  \begin{eqnarray*}
    (I)
    &  =
    & \int_{\bigcup_{i=1}^2 \spt \rho_i (t)}
      \modulo{
      \bar\rho \left(X_1 (\bar t; t, x)\right)
      -
      \bar\rho \left(X_2 (\bar t; t, x)\right)}
      \exp\modulo{\int_{\bar t}^t \div_x V_1 \left(\tau, X_1 (\tau; t, x)\right) \d\tau}
      \d{x}
    \\
    & \leq
    & \int_{\bigcup_{i=1}^2 \spt \rho_i (t)}
      \norma{\grad_x \bar\rho}_{\L\infty (\reali^N; \reali^N)}
      \norma{X_1 (\bar t; t, x) - X_2 (\bar t; t, x)}
    \\
    &
    & \qquad\qquad\times
      \exp \left(\norma{D_x V_1}_{\L\infty (\langle\bar t, t\rangle \times \reali^N; \reali^N)} \modulo{t-\bar t}\right)
      \d{x}
    \\
    & \leq
    & \norma{\grad_x \bar\rho}_{\L\infty (\reali^N; \reali^N)}
    \\
    &
    & \times
      \mathcal{L}^N \!
      \left(
      \spt\bar\rho,
      \max_{i=1,2}
      \norma{V_i}_{\L\infty (\langle\bar t, t\rangle \times \spt\bar\rho; \reali^N)}
      \exp \left(
      \norma{D_x V_i}_{\L\infty (\langle \bar t, t \rangle \times \reali^N; \reali^{N\times N})}
      \modulo{t-\bar t}\right)
      \modulo{t-\bar t}\right)\;
    \\
    &
    & \quad \times
      \norma{V_1-V_2}_{\L1 (\langle \bar t, t\rangle; \L\infty (\reali^N; \reali))}
    \\
    &
    & \quad \times
      \exp \left(
      \left( \norma{D_x V_1}_{\L\infty ((\langle \bar t, t\rangle \times\reali^N; \reali^{N\times N})}
      +
      \norma{D_x V_2}_{\L\infty ((\langle \bar t, t\rangle \times\reali^N; \reali^{N\times N})}\right)
      \modulo{t-\bar t}
      \right) \,.
  \end{eqnarray*}
  Passing to the estimate of $(II)$, using the inequality
  $\modulo{e^a - e^b} \leq e^{\max\{a,b\}} \modulo{a-b}$,
  \begin{eqnarray*}
    &
    & \modulo{
      \exp\left(-\int_{\bar t}^t \div_x V_1\left(\tau, X_1 (\tau; t, x)\right) \d\tau\right)
      -
      \exp\left(-\int_{\bar t}^t \div_x V_2\left(\tau, X_1 (\tau; t, x)\right) \d\tau\right)
      }
    \\
    & \leq
    & \exp\left(
      \max_{i=1,2} \norma{D_x V_i}_{\L\infty (\langle \bar t, t \rangle \times \reali^N; \reali^{N\times N})} \,
      \modulo{t-\bar t}
      \right)
    \\
    &
    & \quad \times
      \modulo{
      \int_{\bar t}^t
      \modulo{
      \div_x V_2\left(\tau, X_1 (\tau; t, x)\right)
      -
      \div_x V_1\left(\tau, X_1 (\tau; t, x)\right)
      } \d\tau
      }
    \\
    & \leq
    & \exp\left(
      \max_{i=1,2}
      \norma{D_x V_i}_{\L\infty (\langle \bar t, t \rangle \times \reali^N; \reali^{N\times N})} \,
      \modulo{t-\bar t}
      \right) \;
      \norma{\div_x V_1 - \div_x V_2}_{\L\infty (\langle\bar t, t\rangle \times \reali^N; \reali)} \;
      \modulo{t-\bar t}
  \end{eqnarray*}
  so that
  \begin{eqnarray*}
    (II)
    & \leq
    & \exp\left(
      \max_{i=1,2}
      \norma{D_x V_i}_{\L\infty (\langle \bar t, t \rangle \times \reali^N; \reali^{N\times N})} \,
      \modulo{t-\bar t}
      \right) \;
      \norma{\div_x (V_1 - V_2)}_{\L\infty (\langle\bar t, t\rangle \times \reali^N; \reali)}
    \\
    &
    &\quad \times
      \norma{\bar\rho}_{\L1 (\reali^N; \reali)} \;
      \modulo{t-\bar t} \,.
  \end{eqnarray*}
  To bound $(III)$, use~\eqref{eq:12} and proceed similarly:
  \begin{eqnarray*}
    &
    & \modulo{
      \exp\left(
      -\int_{\bar t}^t \div_x V_2\left(\tau, X_1 (\tau; t, x)\right) \d\tau
      \right)
      -
      \exp\left(
      -\int_{\bar t}^t \div_x V_2\left(\tau, X_2 (\tau; t, x)\right) \d\tau
      \right)
      }
    \\
    & \leq
    & \exp\left(
      \norma{D_x V_2}_{\L\infty (\langle\bar t, t \rangle \times \reali^N; \reali^{N\times N})}
      \modulo{t-\bar t}\right)
    \\
    &
    & \quad \times
      \modulo{
      \int_{\bar t}^t
      \modulo{
      \div_x V_2\left(\tau, X_1 (\tau; t, x)\right)
      -
      \div_x V_2\left(\tau, X_2 (\tau; t, x)\right)
      }\d\tau
      }
    \\
    & \leq
    & \exp\left(
      2
      \norma{D_x V_2}_{\L\infty (\langle\bar t, t \rangle \times \reali^N; \reali^{N\times N})}
      \modulo{t-\bar t}\right) \;
      \norma{\grad_x \div_x V_2}_{\L\infty (\langle\bar t, t\rangle \times \reali^N; \reali^N)}
    \\
    &
    & \qquad \times
      \norma{V_1-V_2}_{\L1 (\langle \bar t, t\rangle; \L\infty (\reali^N; \reali))} \;
      \modulo{t-\bar t}
  \end{eqnarray*}
  so that
  \begin{eqnarray*}
    (III)
    & \leq
    & \exp\left(
      2\norma{D_x V_2}_{\L\infty (\langle\bar t, t \rangle \times \reali^N; \reali^{N\times N})}
      \modulo{t-\bar t}
      \right)
      \norma{\bar\rho}_{\L1 (\reali^N; \reali)} \;
      \norma{\grad_x \div_x V_2}_{\L\infty (\langle\bar t, t\rangle \times \reali^N; \reali^N)}
    \\
    &
    & \qquad \times
      \norma{V_1-V_2}_{\L1 (\langle \bar t, t\rangle; \L\infty (\reali^N; \reali))} \;
      \modulo{t-\bar t}
  \end{eqnarray*}
  Summing up the expressions obtained:
  \begin{eqnarray*}
    &
    & \norma{\rho_1 (t)-\rho_2 (t)}_{\L1 (\reali^N; \reali)}
    \\
    & \leq
    & \norma{\grad_x \bar\rho}_{\L\infty (\reali^N; \reali^N)} \;
      \norma{V_1-V_2}_{\L1 (\langle \bar t, t\rangle; \L\infty (\reali^N; \reali))}
    \\
    &
    & \quad \times
      \mathcal{L}^N\left(\spt\bar\rho, \max_{i=1,2}
      \norma{V_i}_{\L\infty (\langle\bar t, t\rangle \times \spt\bar\rho)}
      \exp \left(
      \norma{D_x V_i}_{\L\infty (\langle \bar t, t \rangle \times \reali^N)}
      \modulo{t-\bar t}\right)
      \modulo{t-\bar t}\right)
    \\
    &
    & \quad \times
      \exp \left(
      \left( \norma{D_x V_1}_{\L\infty ((\langle \bar t, t\rangle \times\reali^N; \reali^{N\times N})}
      +
      \norma{D_x V_2}_{\L\infty ((\langle \bar t, t\rangle \times\reali^N; \reali^{N\times N})}\right)
      \modulo{t-\bar t}
      \right)
    \\
    &
    & +
      \exp\left[
      \max_{i=1,2} \norma{D_x V_i}_{\L\infty (\langle \bar t, t \rangle \times \reali^N)} \modulo{t-\bar t}
      \right]
      \norma{\div_x (V_1 - V_2)}_{\L\infty (\langle\bar t, t\rangle \times \reali^N; \reali)}
      \norma{\bar\rho}_{\L1 (\reali^N; \reali)}
      \modulo{t-\bar t}
    \\
    &
    & +
      \exp\left(
      2\norma{D_x V_2}_{\L\infty (\langle\bar t, t \rangle \times \reali^N; \reali^{N\times N})}
      \modulo{t-\bar t}
      \right) \;
      \norma{\bar\rho}_{\L1 (\reali^N; \reali)} \;
      \norma{\grad_x \div_x V_2}_{\L\infty (\langle\bar t, t\rangle \times \reali^N; \reali^N)}
    \\
    &
    & \qquad \times
      \norma{V_1-V_2}_{\L1 (\langle \bar t, t\rangle; \L\infty (\reali^N; \reali))}\;
      \modulo{t-\bar t}
  \end{eqnarray*}
  Introduce $C$ as in~\eqref{eq:13}. Then,
  \begin{eqnarray*}
    &
    & \norma{\rho_1 (t)-\rho_2 (t)}_{\L1 (\reali^N; \reali)}
    \\
    & \leq
    & \norma{\grad_x \bar\rho}_{\L\infty (\reali^N; \reali^N)} \;
      \mathcal{L}^N \left(
      \spt \bar\rho, C e^{C \modulo{t-\bar t}}
      \modulo{t-\bar t}
      \right)
      e^{2C \modulo{t-\bar t}}
      \norma{V_1-V_2}_{\L1 (\langle \bar t, t\rangle; \L\infty (\reali^N; \reali))}
    \\
    &
    & +
      \norma{\bar\rho}_{\L1 (\reali^N; \reali)} \;
      e^{2C \modulo{t-\bar t}} \;
      \norma{\div_x (V_1 - V_2)}_{\L\infty (\langle\bar t, t\rangle \times \reali^N; \reali)} \;
      \modulo{t-\bar t}
    \\
    &
    & +
      \norma{\bar\rho}_{\L1 (\reali^N; \reali)} \;
      C\;
      e^{2C \modulo{t-\bar t}} \;
      \norma{V_1-V_2}_{\L1 (\langle \bar t, t\rangle; \L\infty (\reali^N; \reali))}\;
      \modulo{t-\bar t}
  \end{eqnarray*}
  completing the proof.
\end{proof}

\begin{proofof}{Proposition~\ref{prop:0}}
  The first statement follows from~Lemma~\ref{lem:CL1}.  Define
  $V_i (t,x) = v \left(t, x, P_i (t)\right)$, with
  $P_i (t) = \bar P +\int_0^t u_i (\tau) \d\tau$, for $i=1,2$. Then,
  direct computations yield:
  \begin{eqnarray*}
    \norma{V_i}_{\L\infty ([0,t]\times\reali^N; \reali^N)}
    & \leq
    & \norma{v}_{\L\infty ([0,t]\times\reali^N\times \reali^m; \reali^N)} \,.
    \\
    \norma{D_x V_i}_{\L\infty ([0,t]\times\reali^N; \reali^{N\times N})}
    & =
    & \norma{D_x v}_{\L\infty ([0,t]\times\reali^N\times\reali^m; \reali^{N\times N)}} \,.
    \\
    \norma{\grad_x \div_x V_i}_{\L\infty ([0,t]\times\reali^N; \reali^N)}
    & \leq
    & \norma{\grad_x \div_x v}_{\L\infty ([0,t]\times\reali^N\times\reali^m; \reali^N)} \,.
    \\
    \norma{V_1 - V_2}_{\L1 ([0,t]; \L\infty (\reali^N; \reali))}
    & =
    & \int_0^t \sup_{x\in \reali^N}
      \norma{
      v\left(\tau, x, P_1 (\tau)\right) - v\left(\tau, x, P_2 (\tau)\right)
      }
      \d\tau
    \\
    & \leq
    & \norma{D_P v}_{\L\infty ([0,t]\times\reali^N\times\reali^m; \reali^{N\times m})} \,
      t \, \norma{P_1 - P_2}_{\L\infty ([0,t]; \reali^m)} \,.
    \\
    \norma{\div_x (V_1 - V_2)}_{\L\infty ([0,t]\times\reali^N; \reali)}
    & \leq
    & \norma{\div_x v}_{\L\infty ([0,t]\times\reali^N\times\reali^m; \reali)} \,
      \norma{P_1 - P_2}_{\L\infty ([0,t]; \reali^N)} \,.
  \end{eqnarray*}
  Now, \eqref{eq:29} directly follows from~\eqref{eq:12} in
  Lemma~\ref{lem:ode}. To prove~\eqref{eq:34} use Lemma~\ref{lem:CL2}.
\end{proofof}

We recall here, without proof, the following result about G\^ateaux
and Fréchet differentiability for later use.

\begin{lemma}[{\cite[Lemma~1.15]{Schwartz}}]
  \label{lem:A2}
  Let $X,Y$ be Banach spaces, $A \subseteq X$ be open, $x_o \in A$ and
  $J \colon A \to Y$ be a map. Assume
  \begin{enumerate}
  \item $J$ is G\^ateaux differentiable at all $x \in A$ in all
    directions $v \in X$;
  \item the map $v \to D_vJ (x)$ is linear and continuous for all
    $x \in A$;
  \item
    $\lim\limits_{x \to x_o} \sup\limits_{\norma{v}_X = 1} \norma{D_v
      J (x) - D_v J (x_o)}_Y = 0$.
  \end{enumerate}
  \noindent Then, $J$ is Fréchet differentiable at $x_o$.
\end{lemma}

The next result describes the Fréchet differentiability of the
characteristic curves.
\begin{lemma}
  \label{lem:DX}
  Fix $t \in \left[0,T\right[$, $\Delta t \in \left]0, T-t \right]$
  and $x \in \reali^N$. If
  $v \in \C2 ([0,T] \times \reali^N \times \reali^N; \reali^N)$, then
  the map
  \begin{displaymath}
    \begin{array}[b]{ccccc}
      \mathcal{X}_{t,x}
      & \colon
      & \reali^N
      & \to
      & \C0 ([t, t+\Delta t]; \reali^N)
      \\
      &
      & w
      & \to
      & \mathcal{X}_{t,x} (w) \,,
    \end{array}
  \end{displaymath}
  defined so that
  $\tau \to \left(\mathcal{X}_{t,x} (w)\right) \! (\tau)$ solves the
  Cauchy problem
  \begin{equation}
    \label{eq:30}
    \left\{
      \begin{array}{@{}l@{\,}}
        \xi' = v \left(t, \xi, P (t) + (\tau-t)w\right)
        \\
        \xi (t) = x \,,
      \end{array}
    \right.
  \end{equation}
  is Fréchet differentiable in $\reali^N$. Moreover
  $\mathcal{X}_{t,x}$ has the Taylor expansion
  \begin{displaymath}
    \mathcal{X}_{t,x} (w + \delta_w) = \mathcal{X}_{t,x} (w) +
    D\mathcal{X}_{t,x} (w) \, \delta_w + o (\delta_w)
    \quad \mbox{in } \C0 \mbox{ as } \delta_w \to 0
  \end{displaymath}
  where $\tau \to \left(D\mathcal{X}_{t,x} (w)\right) (\tau)$ solves
  the linear first order $N \times N$ matrix differential
  equation
  \begin{equation}
    \label{eq:17}
    \left\{
      \begin{array}{l}
        Y'
        =
        D_x v\left(
        t,\mathcal{X}_{t,x} (w) (\tau), P (t) + (\tau-t)w
        \right)
        Y
        \vspace{.2cm}\\
        \qquad\qquad
        +
        (\tau - t) \,
        D_P v\left(
        t,
        \mathcal{X}_{t,x} (w) (\tau),
        P (t)+ (\tau-t) w
        \right)
        \vspace{.2cm}\\
        Y (t) = 0
      \end{array}
    \right.
  \end{equation}
  and the term $D\mathcal{X}_{t,x} (w)$ satisfies the expansion, as
  $\tau \to t$,
  \begin{equation}
    \label{eq:19}
    \left(D\mathcal{X}_{t,x} (w)\right) (\tau)
    =
    \dfrac{(\tau-t)^2}{2} \,
    D_P v\left(t, x, P (t)\right)  + o (\tau-t)^2 \,.
  \end{equation}
\end{lemma}

\begin{proof}
  Since $t$ and $x$ are kept fixed throughout this proof, we write
  $\mathcal{X} (w)$ for $\mathcal{X}_{t,x} (w)$.  Recall that, for
  $\tau \in [t, t + \Delta t]$,
  \begin{displaymath}
    \mathcal{X} (w) (\tau)
    =
    x
    +
    \int_t^\tau
    v\left(
      t, \mathcal{X} (w) (s), P (t) + (s-t)w
    \right) \d{s} \,.
  \end{displaymath}
  Fix a direction $\delta_w \in \reali^N \setminus\{0\}$.  First we
  show the boundedness of the difference quotient
  \begin{equation*}
    \frac{\norma{\mathcal{X} (w+\epsilon\delta_w) (\tau)
        - \mathcal{X} (w) (\tau)}}{\epsilon}.
  \end{equation*}
  For $\tau \in [t , t + \Delta t]$, we have
  \begin{align*}
    &
      \frac{1}{\epsilon}
      \norma{\mathcal{X} (w+\epsilon\delta_w) (\tau) - \mathcal{X} (w) (\tau)}
    \\
    \le
    & \frac{1}{\epsilon}\int_t^\tau
      \norma{
      v\left(t,\mathcal{X} (w+\epsilon\delta_w) (s), P (t)
      + (s-t) (w+\epsilon\delta_w)\right)
      {-}
      v\left(t, \mathcal{X} (w) (s), P (t) + (s-t)w\right)
      } \! \d{s}
    \\
    \le
    & \frac{1}{\epsilon} \int_t^\tau
      \bigl\|
      v\left(t,\mathcal{X} (w+\epsilon\delta_w) (s), \!P (t)
      \!+ \! (s-t) (w+\epsilon\delta_w)\right)
      \!-\!
      v\left(t, \mathcal{X} (w) (s),\! P (t) \!+\!
      (s-t)(w+\epsilon\delta_w)\right)
      \bigr\|\! \d{s}
    \\
    & + \frac{1}{\epsilon}
      \int_t^\tau
      \norma{
      v\left(t,\mathcal{X} (w) (s), P (t) + (s-t) (w+\epsilon\delta_w)\right)
      -
      v\left(t, \mathcal{X} (w) (s), P (t) + (s-t)w\right)
      } \d{s}
    \\
    \le
    & \norma{v}_{\C1\left([0,T] \times \reali^N \times \reali^N; \reali^N\right)}
      \left[
      \int_t^\tau
      \frac{\norma{\mathcal{X} (w+\epsilon\delta_w) (s) - \mathcal{X} (w) (s)}}
      {\epsilon}
      \d{s}
      +
      \int_t^\tau (s-t) \norma{\delta_w} \d{s}
      \right]
    \\
    \le
    & \norma{v}_{\C1\left([0,T] \times \reali^N \times \reali^N; \reali^N\right)}
      \left[
      \int_t^\tau
      \frac{\norma{\mathcal{X} (w+\epsilon\delta_w) (s) - \mathcal{X} (w) (s)}}
      {\epsilon}
      \d{s}
      +
      \left(\Delta t\right)^2 \norma{\delta_w}
      \right].
  \end{align*}
  Hence an application of Gr\"onwall Lemma, see, e.g.,
  \cite[Chapter~3, Lemma~3.1]{BressanPiccoliBook} ensures that
  \begin{equation}
    \label{eq:bound_diff_quot}
    \frac{\norma{\mathcal{X} (w+\epsilon\delta_w) (\tau)
        - \mathcal{X} (w) (\tau)}}
    {\epsilon} \le K_1
    \left(\Delta t\right)^3 \norma{\delta_w}
    \exp\left(K_1 \Delta t\right),
  \end{equation}
  where
  $K_1 = \norma{v}_{\C1\left([0,T] \times \reali^N \times \reali^N;
      \reali^N\right)}$. Consequently
  \begin{equation}
    \label{eq:27}
    \lim_{\epsilon \to 0} \sup_{\tau \in [t, t+\Delta t]}
    \norma{\mathcal{X} (w+\epsilon\delta_w) (\tau) - \mathcal{X} (w) (\tau)}
    = 0 \,.
  \end{equation}
  We now prove the existence of directional derivatives of
  $\mathcal{X}$ along the direction
  $\delta_w \in \reali^N \setminus\{0\}$.  Calling $\tau \to Y (\tau)$
  the solution to the Cauchy problem~\eqref{eq:17}, we have
  \begin{eqnarray*}
    &
    & \dfrac{\mathcal{X} (w + \epsilon \delta_w) (\tau)
      -
      \mathcal{X} (w) (\tau)}{\epsilon}
      -
      Y (\tau) \; \delta_w
    \\
    & =
    & \! \dfrac{1}{\epsilon} \!
      \int_t^\tau \!\!
      \left[
      v \! \left(
      t,
      \mathcal{X} (w+\epsilon \delta_w) (s),
      P (t) + (s-t)(w+\epsilon \delta_w)
      \right)
      -
      v \! \left(t, \mathcal{X} (w) (s), P (t) + (s-t)w\right)
      \!\right] \!
      \d{s}
    \\
    &
    & -
      \int_t^\tau
      D_xv\left(t, \mathcal{X} (w) (s), P (t) + (s-t)w\right)
      Y (s) \d{s} \, \delta_w
    \\
    &
    & -
      \int_t^\tau
      (s-t)
      D_P v \left(t, \mathcal{X} (w) (s), P (t) + (s-t)w\right)
      \d{s} \; \delta_w
    \\
    & =
    & \dfrac{1}{\epsilon}
      \int_t^\tau
      \bigl[
      v \! \left(
      t,
      \mathcal{X} (w+\epsilon \delta_w) (s),
      P (t) + (s-t)(w+\epsilon \delta_w)
      \right)
    \\
    &
    & \qquad\qquad\qquad
      -
      v \! \left(t, \mathcal{X} (w) (s), P (t) + (s-t)(w+\delta_w)\right)
      \bigr]
      \d{s}
    \\
    &
    & -
      \int_t^\tau
      D_xv\left(t, \mathcal{X} (w) (s), P (t) + (s-t)w\right)
      Y (s) \d{s} \, \delta_w
    \\
    &
    & +
      \dfrac{1}{\epsilon}
      \int_t^\tau \!\!
      \left[
      v \! \left(
      t,
      \mathcal{X} (w) (s),
      P (t) + (s-t)(w+\epsilon \delta_w)
      \right)
      -
      v \! \left(t, \mathcal{X} (w) (s), P (t) + (s-t)w\right)
      \!\right] \!
      \d{s}
    \\
    &
    & -
      \int_t^\tau
      (s-t)
      D_P v \left(t, \mathcal{X} (w) (s), P (t) + (s-t)w\right)
      \d{s} \; \delta_w
    \\
    & =
    & \int_t^\tau \int_0^1
      \Bigl(
      D_x
      v\left(
      t,
      \theta \mathcal{X} (w+\epsilon \delta_w) (s)
      +
      (1-\theta) \mathcal{X} (w) (s),
      P (t) + (s-t)(w+\epsilon\delta_w)
      \right) \d\theta
    \\
    &
    & \qquad\qquad\qquad
      \times
      \dfrac{\mathcal{X} (w+\epsilon \delta_w) (s)
      -
      \mathcal{X} (w) (s)}{\epsilon}
      \Bigr)
      \d s
    \\
    &
    & -
      \int_t^\tau
      D_xv\left(t, \mathcal{X} (w) (s), P (t) + (s-t)w\right)
      Y (s) \d{s} \, \delta_w
    \\
    &
    & +
      \int_t^\tau (s-t)
      \Bigl(
      \int_0^1
      D_P
      v\left(
      t,
      \mathcal{X} (w) (s),
      P (t)+ (s-t) (w + (1-\theta)\epsilon\delta_w)
      \right) \d\theta
    \\
    &
    &
      \qquad\qquad\qquad\qquad
      -
      D_P v \left(t, \mathcal{X} (w) (s), P (t) + (s-t)w\right)
      \Bigr)
      \d{s} \; \delta_w
    \\
    & =
    & \int_t^\tau \int_0^1
      \Bigl(
      D_x
      v\left(
      t,
      \theta \mathcal{X} (w+\epsilon \delta_w) (s)
      +
      (1-\theta) \mathcal{X} (w) (s),
      P (t) + (s-t)(w+\epsilon\delta_w)
      \right) \d\theta
    \\
    &
    & \qquad\qquad\qquad
      \times
      \dfrac{\mathcal{X} (w+\epsilon \delta_w) (s)
      -
      \mathcal{X} (w) (s)}{\epsilon}
      \Bigr)
      \d s
    \\
    &&
       \mp \int_t^\tau
       D_x
       v\left(
       t,
       \mathcal{X} (w) (s),
       P (t) + (s-t)w
       \right)
       \dfrac{\mathcal{X} (w+\epsilon \delta_w) (s)
       -
       \mathcal{X} (w) (s)}{\epsilon}
       \d s
    \\
    &
    & -
      \int_t^\tau
      D_xv\left(t, \mathcal{X} (w) (s), P (t) + (s-t)w\right)
      Y (s) \d{s} \, \delta_w
    \\
    &
    & +
      \int_t^\tau (s-t)
      \Bigl(
      \int_0^1
      D_P
      v\left(
      t,
      \mathcal{X} (w) (s),
      P (t)+ (s-t) (w + (1-\theta)\epsilon\delta_w)
      \right) \d\theta
    \\
    &
    &
      \qquad\qquad\qquad\qquad
      -
      D_P v \left(t, \mathcal{X} (w) (s), P (t) + (s-t)w\right)
      \Bigr)
      \d{s} \; \delta_w
    \\
    & =
    & \int_t^\tau \int_0^1
      \Bigl(
      D_x
      v\left(
      t,
      \theta \mathcal{X} (w+\epsilon \delta_w) (s)
      +
      (1-\theta) \mathcal{X} (w) (s),
      P (t) + (s-t)(w+\epsilon\delta_w)
      \right) \d\theta
    \\
    &
    & \qquad\qquad\qquad
      -
      D_x v\left(t, \mathcal{X} (w) (s), P (t) + (s-t)w\right)
      \Bigr)
      \frac{      \mathcal{X} (w+\epsilon \delta_w) (s)
      -
      \mathcal{X} (w) (s)
      }{\epsilon}
      \d{s}
    \\
    &
    & +
      \int_t^\tau
      D_xv\left(t, \mathcal{X} (w) (s), P (t) + (s-t)w\right)
    \\
    &
    & \qquad\qquad\qquad
      \times
      \left(\dfrac{\mathcal{X} (w+\epsilon \delta_w) (s)
      -
      \mathcal{X} (w) (s)}{\epsilon}
      -
      Y (s) \delta_w \right)
      \d{s}
    \\
    &
    & +
      \int_t^\tau (s-t)
      \Bigl(
      \int_0^1
      D_P
      v\left(
      t,
      \mathcal{X} (w) (s),
      P (t)+ (s-t) (w + (1-\theta)\epsilon\delta_w)
      \right) \d\theta
    \\
    &
    &
      \qquad\qquad\qquad\qquad
      -
      D_P v \left(t, \mathcal{X} (w) (s), P (t) + (s-t)w\right)
      \Bigr)
      \d{s} \; \delta_w
  \end{eqnarray*}
  Calling $\O$ a constant dependent on the $\C2$ norm of $v$ and on
  the right hand side of~(\ref{eq:bound_diff_quot}), the above
  equality leads to
  \begin{eqnarray*}
    &
    & \norma{\dfrac{\mathcal{X} (w + \epsilon \delta_w) (\tau)
      -
      \mathcal{X} (w) (\tau)}{\epsilon}
      -
      Y (\tau) \; \delta_w}
    \\
    & \leq
    & \O +
      \int_t^\tau
      \O \, \norma{\dfrac{\mathcal{X} (w + \epsilon \delta_w) (\tau)
      -
      \mathcal{X} (w) (\tau)}{\epsilon}
      -
      Y(\tau) \; \delta_w} \d{s}
    \\
    &
    & +
      \int_t^\tau
      \O \, (s-t) \, \epsilon \d{s} \, \delta_w \,.
  \end{eqnarray*}
  Thanks to~\eqref{eq:27}, an application of Gr\"onwall Lemma proves
  the directional differentiability of $w \to \mathcal{X}(w)$ in the
  direction $\delta_w$.

  To prove the differentiability of $\mathcal{X}$, we are left to
  verify that~2.~and~3.~in Lemma~\ref{lem:A2} hold. The linearity of
  $\delta_w \to D\mathcal{X} (w) (\delta_w)$ is immediate, thanks to
  the homogeneous initial datum in~\eqref{eq:17}. The assumed $\C2$
  regularity of $v$ ensures the $\C1$ regularity of the right hand
  side in~\eqref{eq:17} and, hence, the boundedness of
  $\delta_w \to D\mathcal{X} (w) (\delta_w)$ (in the sense of linear
  operators), completing the proof of~2. Standard theorems on the
  continuous dependence of solutions to ordinary differential
  equations from parameters, see, e.g.,
  \cite[Theorem~4.2]{BressanPiccoliBook}, ensure that also~3.~in
  Lemma~\ref{lem:A2} holds, completing the proof of the
  differentiability of $\mathcal{X}$.

  The proof of the Taylor expansion~(\ref{eq:19}) follows easily
  using~(\ref{eq:17}). Indeed, by~(\ref{eq:17}), we deduce that
  $Y(t) = Y'(t) = 0$, while $Y''(t) = D_P v \left(t, x, P (t)\right)$,
  so that, if $\tau \in [t, t + \Delta t]$, then
  \begin{equation*}
    Y(\tau) = \frac{\left(\tau - t\right)^2}{2} D_P v \left(t, x, P (t)\right)
    + o \left(\left(\tau - t\right)^2\right).
  \end{equation*}
  This completes the proof of~(\ref{eq:19}) and of the lemma.
  % Direct computations yield:
  % \begin{eqnarray*}
  %   \left(\mathcal{X}_{t,x} (w)\right) (t)
  %   & =
  %   & x \,,
  %   \\
  %   \partial_\tau \left(\left(\mathcal{X}_{t,x} (w)\right) (\tau)\right)_{|\tau=t}
  %   & =
  %   & 0 \,,
  %   \\
  %   \partial^2_\tau \left(\left(\mathcal{X}_{t,x} (w)\right) (\tau)\right)_{|\tau=t}
  %   & =
  %   & D_P v \left(t, x, P (t)\right) \,.
  % \end{eqnarray*}
  % completing the proof of~\eqref{eq:19}.
\end{proof}

\begin{proofof}{Lemma~\ref{lem:1}}
  The first statement is a direct consequence of Lemma~\ref{lem:CL1}.
  To prove~\eqref{eq:3}, we apply~Lemma~\ref{lem:CL2} with
  $V (\tau,x) = v (t,x,P (t) + (\tau - t)\, w)$ for
  $\tau \in [t, t+\Delta t]$:
  \begin{eqnarray*}
    \norma{V_i}_{\L\infty ([t, t+\Delta t] \times \reali^N; \reali^N)}
    & \leq
    & \norma{v}_{\L\infty ([t, t+\Delta t] \times \reali^N \times B (P (t), \Delta t \norma{w_i}); \reali^N)} \,,
    \\
    \norma{D_x V_i}_{\L\infty ([t, t+\Delta t] \times \reali^N; \reali^{N\times N})}
    & \leq
    & \norma{D_x v}_{\L\infty ([t, t+\Delta t] \times \reali^N \times B (P (t), \Delta t \norma{w_i}),\reali^{N\times N})} \,,
    \\
    \norma{V_1 - V_2}_{\L1 ([t. t+\Delta t], \L\infty(\reali^N; \reali^N))}
    & \leq
    & \frac12 \;
      \norma{D_P v}_{\L\infty ([t, t+\Delta t] \times \reali^N \times \reali^N; \reali^{N\times N})}
      (\Delta t)^2 \norma{w_1 - w_2} \,,
    \\
    \norma{\div_x (V_1 - V_2)}_{\L\infty ([t. t+\Delta t] \times \reali^N; \reali^N)}
    & \leq
    & \norma{D_P \div_x v}_{\L\infty ([t, t+\Delta t] \times \reali^N \times \reali^N; \reali^N)}
      (\Delta t) \, \norma{w_1 - w_2} \,.
  \end{eqnarray*}
  With the notation~\eqref{eq:14}, and assuming that $C\geq 1$,
  \begin{eqnarray*}
    &
    & \norma{\rho_1 (t+\Delta t) - \rho_2 (t+\Delta t)}_{\L1 (\reali^N; \reali)}
    \\
    & \leq
    & \norma{\grad_x \bar\rho}_{\L\infty (\reali^N; \reali^N)} \;
      \mathcal{L}^N \left(
      \spt \bar\rho, C e^{C \Delta t}
      \Delta t
      \right)
      e^{2C \Delta t}
      C (\Delta t)^2 \norma{w_1 - w_2}
    \\
    &
    & +
      \norma{\bar\rho}_{\L1 (\reali^N; \reali)} \; e^{2C \Delta t} \;
      C (\Delta t)^2 \norma{w_1 - w_2}    \\
    &
    & +
      \norma{\bar\rho}_{\L1 (\reali^N; \reali)} \;
      C^2 \,
      e^{2C \Delta t} \;
      (\Delta t)^3 \norma{w_1 - w_2}
    \\
    & \leq
    & \left(
      \norma{\grad_x \bar\rho}_{\L\infty (\reali^N; \reali^N)} \;
      \mathcal{L}^N \left(
      \spt \bar\rho, C e^{C \Delta t}
      \Delta t
      \right)
      +
      (1+C\,\Delta t) \norma{\bar\rho}_{\L1 (\reali^N; \reali)}
      \right)
    \\
    &
    & \quad
      \times
      C e^{2C \Delta t} \, (\Delta t)^2 \, \norma{w_1 - w_2}
  \end{eqnarray*}
  completing the proof.
\end{proofof}

\begin{proofof}{Proposition~\ref{prop:1}}
  The map $\mathcal{J}_{t, \Delta t}$ is well defined by
  Lemma~\ref{lem:1}.  To prove its Lipschitz continuity, let
  $w_1, w_2 \in \reali^N$.  Denote
  $V_i (\tau, x) = v (t, x, P (t) + (\tau-t)w_i)$; $X_i = X_{t,w_i}$
  the solution to~\eqref{eq:10} and $\rho_i = \rho_{w_i}$ the
  corresponding solution to~\eqref{eq:6}.  Straightforward
  computations yield
  \begin{eqnarray*}
    \modulo{\mathcal{J}_{t, \Delta t} (w_1) - \mathcal{J}_{t, \Delta t} (w_2)}
    & \leq
    &
      \int_{\reali^n}
      \modulo{\rho_1 (t + \Delta t, x) - \rho_2 (t+\Delta t, x)}
      \modulo{\psi (x)} \, \d{x}
    \\
    & \leq
    & \norma{\rho_1 (t+\Delta t) - \rho_2 (t+\Delta t)}_{\L1 (\reali^N; \reali)} \;
      \norma{\psi}_{\L\infty (\reali^N; \reali)} \,.
  \end{eqnarray*}
  and the proof is completed thanks to~\eqref{eq:3}.
\end{proofof}

\begin{proofof}{Theorem~\ref{thm:bo}}
  Recall~\eqref{eq:7}--\eqref{eq:28}. Fix $t$ and $t+\Delta t$ in
  $[0,T]$. The solution $\tau \to \mathcal{X}_w (\tau;t+\Delta t, x)$
  to
  \begin{equation}
    \label{eq:31}
    \left\{
      \begin{array}{@{}l@{\,}}
        \xi' = v \left(t, \xi, P (t) + (\tau-t)w\right)
        \\
        \xi (t+\Delta t) = x
      \end{array}
    \right.
    \qquad \tau \in [t, t+\Delta t]
  \end{equation}
  will be shortened to $\tau \to \mathcal{X}_w (\tau;x)$.  By
  Lemma~\ref{lem:DX}, we have the expansion
  \begin{equation}
    \label{eq:16}
    \mathcal{X}_{w+\epsilon\delta_w} (\tau;x)
    =
    \mathcal{X}_w (\tau;x)
    +
    \epsilon \, D_w \mathcal{X}_w (\tau;x) \, \delta_w + o (\epsilon)
    \quad \mbox{ in } \C0 \mbox{ as } \epsilon \to 0 \,,
  \end{equation}
  where $\tau \to D_w \mathcal{X}_w (\tau;t+\Delta t, x)$, or
  $\tau \to D_w \mathcal{X}_w (\tau;x)$ for short, solves the Cauchy
  Problem
  \begin{equation}
    \label{eq:32}
    \!\!\!\!\!\!\!
    \left\{
      \begin{array}{@{}l@{\!}}
        Y'
        =
        D_x v \! \left(
        t, \mathcal{X}_w (\tau;x), P (t) + (\tau-t)w
        \right)
        Y
        +
        (\tau - t)
        D_P v \! \left(
        t,
        \mathcal{X}_w (\tau;x),
        P (t)+ (\tau-t) w
        \right)
        \\
        Y (t+\Delta t) = 0
      \end{array}
    \right.
  \end{equation}
  for $\tau \in [t, t+\Delta t]$.  With reference to~\eqref{eq:33},
  denote for simplicity $\mathcal{J} = \mathcal{J}_{t, \Delta t}$,
  $\psi = \psi$ and compute:
  \begin{eqnarray*}
    &
    & \dfrac{1}{\epsilon}
      \left(\mathcal{J} (w+\epsilon\delta_w) - \mathcal{J} (w)\right)
    \\
    & =
    & \dfrac{1}{\epsilon}
      \int_{\reali^N}
      \left(
      \rho_{w+\epsilon\delta_w} (t+\Delta t, x) - \rho_w (t+\Delta t, x)
      \right)
      \psi (x)
      \d{x}
    \\
    & =
    & \dfrac{1}{\epsilon}
      \int_{\reali^N}
      \Big[
      \rho \left(t, \mathcal{X}_{w+\epsilon\delta_w} (t;x)\right)
    \\
    &
    & \qquad \qquad \times
      \exp \left(
      -\int_t^{t+\Delta t}
      \div_x v
      \left(
      s, \mathcal{X}_{w+\epsilon\delta_w}(s;x), P (t) + (s-t)(w+\epsilon\delta_w)
      \right) \d{s}
      \right)
    \\
    &
    & \qquad
      -
      \rho \left(t, \mathcal{X}_w (t;x)\right)
    \\
    &
    & \qquad \qquad \times
      \exp \left(
      -\int_t^{t+\Delta t}
      \div_x  v\left(s, \mathcal{X}_w(s;x), P (t) + (s-t)w\right) \d{s}
      \right)
      \Big]
      \psi (x)
      \d{x}
    \\
    & =
    & (I) + (II) + (III)
  \end{eqnarray*}
  where
  \begin{eqnarray*}
    (I)
    & =
    & \dfrac{1}{\epsilon}
      \int_{\reali^N}
      \left[
      \rho \left(t, \mathcal{X}_{w+\epsilon\delta_w} (t;x)\right)
      -
      \rho \left(t, \mathcal{X}_w (t;x)\right)
      \right]
    \\
    &
    & \qquad \times
      \exp \! \left( \!
      -\int_t^{t+\Delta t}
      \div_x v
      \left(
      s, \mathcal{X}_{w+\epsilon\delta_w}(s;x), P (t) + (s-t)(w+\epsilon\delta_w)
      \right) \d{s}
      \right) \! \psi (x) \d{x}
    \\
    (II)
    & =
    & \dfrac{1}{\epsilon}
      \int_{\reali^N}
      \rho \left(t, \mathcal{X}_w (t;x)\right)
    \\
    &
    & \qquad \times
      \Big[
      \exp \left(
      -\int_t^{t+\Delta t}
      \div_x v
      \left(
      s, \mathcal{X}_{w+\epsilon\delta_w}(s;x), P (t) + (s-t)(w+\epsilon\delta_w)
      \right) \d{s}
      \right)
    \\
    &
    & \qquad\qquad-
      \exp \left(
      -\int_t^{t+\Delta t}
      \div_x v\left(s, \mathcal{X}_{w+\epsilon\delta_w}(s;x), P (t) + (s-t)w\right) \d{s}
      \right)
      \Big]
      \psi (x)
      \d{x}
    \\
    (III)
    & =
    & \dfrac{1}{\epsilon}
      \int_{\reali^N}
      \rho \left(t, \mathcal{X}_w (t;x)\right)
    \\
    &
    & \qquad \times
      \Big[
      \exp \left(
      -\int_t^{t+\Delta t}
      \div_x v
      \left(
      s, \mathcal{X}_{w+\epsilon\delta_w}(s;x), P (t) + (s-t) w
      \right) \d{s}
      \right)
    \\
    &
    & \qquad\qquad-
      \exp \left(
      -\int_t^{t+\Delta t}
      \div_x v\left(s, \mathcal{X}_w(s;x), P (t) + (s-t)w\right) \d{s}
      \right)
      \Big]
      \psi (x)
      \d{x}
  \end{eqnarray*}
  The following estimate uses $D_w\mathcal{X}_w$ as defined
  in~\eqref{eq:32} and is of use to compute~$(I)$:
  \begin{eqnarray*}
    &
    & \dfrac{1}{\epsilon}
      \left(
      \rho \left(t, \mathcal{X}_{w+\epsilon\delta_w} (t;x)\right)
      -
      \rho \left(t, \mathcal{X}_w (t;x)\right)
      \right)
    \\
    & =
    & \int_0^1
      \grad_x \rho \left(
      t,
      \theta \mathcal{X}_{w+\epsilon\delta_w} (t;x)
      +
      (1-\theta) \mathcal{X}_w (t;x)
      \right)
      \d\theta \;
      \dfrac{
      \mathcal{X}_{w+\epsilon\delta_w} (t;x)
      -
      \mathcal{X}_w (t;x)}{\epsilon}
    \\
    & \stackrel{\epsilon\to 0}{\to}
    & \grad_x \rho \left(t, \mathcal{X}_w (t;x) \right)\;
      D_w \mathcal{X}_w (t;x) \; \delta_w \,,
  \end{eqnarray*}
  so that,
  \begin{eqnarray*}
    (I)
    & \stackrel{\epsilon\to 0}{\to}
    & \int_{\reali^N}
      \grad_x \rho \left(t, \mathcal{X}_w (t;x)\right)\;
      D_w \mathcal{X}_w (t;x) \, \delta_w
    \\
    &
    & \qquad\qquad
      \times \exp \left(
      -\int_t^{t+\Delta t}
      \div_x v\left(s, \mathcal{X}_w(s;x), P (t) + (s-t)w\right) \d{s}
      \right)
      \psi (x)
      \d{x}
      % \\
      % & =
      % & \dfrac{1}{2} \, (\Delta t)^2 \,
      % \int_{\reali^N}
      % \grad_x \rho (t, x)\;
      % D_P v \left(t,x, P (t)\right)\, \psi_t (x)
      % \d{x} \; \delta_w
      % +
      % o (\Delta t)^2
  \end{eqnarray*}
  while
  \begin{eqnarray*}
    (II)
    & \stackrel{\epsilon\to 0}{\to}
    & -
      \int_{\reali^N}
      \rho\left(t, \mathcal{X}_w (t;x)\right)
      % \\
      % &
      % & \qquad \times
          \, \exp\left(
          -\int_t^{t+\Delta t}
          \div_x v\left(
          s,
          \mathcal{X}_w (s;x),
          P (t) + (s-t)w
          \right)
          \d{s}
          \right)
    \\
    &
    & \qquad \times
      \int_t^{t+\Delta t}
      \grad_P \div_x v\left(
      s,
      \mathcal{X}_w (s;x),
      P (t) + (s-t) w
      \right)
      (s-t)
      \d{s}
      \delta_w
      \psi (x)
      \d{x}
      % \\
      % & =
      % & -
      % \dfrac{1}{2} \, (\Delta t)^2 \,
      % \int_{\reali^N}
      % \rho (t, x) \,
      % \grad_P \div_x v \left(s, x, P (t)\right) \,
      % \psi_t (x) \,
      % \d{x} \,
      % \delta_w
      % +
      % o (\Delta t)^2
  \end{eqnarray*}
  and similarly, using $D_w \mathcal{X}_w$ defined as solution
  to~\eqref{eq:32},
  \begin{eqnarray*}
    (III)
    & \stackrel{\epsilon\to 0}{\to}
    & -\int_{\reali^n}
      \rho\left(t, \mathcal{X}_w (s;x)\right)
      % \\
      % &
      % & \quad \times
          \, \exp\left(
          -\int_t^{t+\Delta t}
          \div_x v\left(
          s,
          \mathcal{X}_w (s;x),
          P (t) + (s-t)w
          \right)
          \d{s}
          \right)
    \\
    &
    & \quad \times
      \int_t^{t+\Delta t}
      \grad_x \div_x v\left(
      s,
      \mathcal{X}_w (s;x),
      P (t) + (s-t) w
      \right)
      D_w \mathcal{X}_w(s;x)
      \d{s}
      \psi (x)
      \d{x} \, \delta_w
      % \\
      % & =
      % & o (\Delta t)^2
  \end{eqnarray*}
  Adding the three terms we get:
  \begin{eqnarray*}
    &
    & \lim_{\epsilon \to 0}
      \dfrac{1}{\epsilon} \,
      \left(\mathcal{J} (w+\epsilon\delta_w) - \mathcal{J} (w)\right)
    \\
    & =
    & \int_{\reali^N}
      \Big[
      \grad_x \rho \left(t, \mathcal{X}_w (t;x)\right)\;
      D_w \mathcal{X}_w(t;x)
    \\
    &
    & \qquad \quad-
      \rho\left(t, \mathcal{X}_w (t;x)\right)
      % \\
      % &
      % & \qquad \qquad
      % \times
          \int_t^{t+\Delta t}
          \Big(
          \grad_P \div_x v\left(
          s,
          \mathcal{X}_w (s;x),
          P (t) + (s-t) w
          \right)
          (s-t)
    \\
    &
    & \qquad\qquad\qquad
      +
      \grad_x \div_x v\left(
      s,
      \mathcal{X}_w (s;x),
      P (t) + (s-t) w
      \right)
      D_w \mathcal{X}_w (s;x)
      \Big)
      \d{s}
      \Big]
    \\
    &
    & \quad\times
      \exp\left(
      -\int_t^{t+\Delta t}
      \div_x v\left(
      s,
      \mathcal{X}_w (s;x),
      P (t) + (s-t)w
      \right)
      \d{s}
      \right)
      \psi (x)
      \d{x} \, \delta_w \,.
  \end{eqnarray*}
  To compute the limit as $\Delta t \to 0$ of the expression above,
  recall that as $\Delta t \to 0$,
  \begin{displaymath}
    \begin{array}{rcl@{\qquad}l}
      \mathcal{X}_w (t; t+\Delta t, x)
      & =
      & x - v\left(t,x,P (t)\right) \, \Delta t + o (\Delta t)
      & [\mbox{by~\eqref{eq:31}}]
      \\
      D_w \mathcal{X}_w (t; t+\Delta t, x)
      & =
      & \frac12 \, D_P v \left(t, x, P (t)\right)
        \, (\Delta t)^2
        + o (\Delta t)^2
      & [\mbox{by~\eqref{eq:19}]}
    \end{array}
  \end{displaymath}
  so that
  \begin{eqnarray*}
    &
    & \lim_{\epsilon \to 0}
      \dfrac{1}{\epsilon} \,
      \left(\mathcal{J} (w+\epsilon\delta_w) - \mathcal{J} (w)\right)
    \\
    & =
    & \dfrac{(\Delta t)^2}{2}
      \int_{\reali^N}
      \left(
      \grad_x \rho (t, x)\,
      D_P v \left(t,x, P (t)\right)
      -
      \rho (t, x) \,
      \grad_P \div_x v \left(s, x, P (t)\right)
      \right)
      \psi (x) \,
      \d{x} \,
      \delta_w
    \\
    &
    & +
      o (\Delta t)^2
  \end{eqnarray*}
  completing the proof.
\end{proofof}

\smallskip

\noindent\textbf{Acknowledgments:} Part of this work was supported by the  PRIN~2015 project \emph{Hyperbolic Systems of Conservation Laws
  and Fluid Dynamics: Analysis and Applications} and by the
GNAMPA~2017 project \emph{Conservation Laws: from Theory to
  Technology}. The \emph{IBM Power Systems Academic Initiative}
substantially contributed to the the numerical integrations.

{\small

  \bibliographystyle{abbrv}

  %\bibliography{miope4.bib}

\begin{thebibliography}{10}

\bibitem{BellomoDogbe}
N.~Bellomo and C.~Dogbe.
\newblock On the modeling of traffic and crowds: a survey of models,
  speculations, and perspectives.
\newblock {\em SIAM Rev.}, 53(3):409--463, 2011.

\bibitem{BressanPiccoliBook}
A.~Bressan and B.~Piccoli.
\newblock {\em Introduction to the mathematical theory of control}, volume~2 of
  {\em AIMS Series on Applied Mathematics}.
\newblock American Institute of Mathematical Sciences (AIMS), Springfield, MO,
  2007.

\bibitem{brezis}
H.~Brezis.
\newblock {\em Functional analysis, {S}obolev spaces and partial differential
  equations}.
\newblock Universitext. Springer, New York, 2011.

\bibitem{MR3110059}
M.~Caponigro, M.~Fornasier, B.~Piccoli, and E.~Tr\'elat.
\newblock Sparse stabilization and optimal control of the {C}ucker-{S}male
  model.
\newblock {\em Math. Control Relat. Fields}, 3(4):447--466, 2013.

\bibitem{MR3285315}
M.~Caponigro, M.~Fornasier, B.~Piccoli, and E.~Tr\'elat.
\newblock Sparse stabilization and control of alignment models.
\newblock {\em Math. Models Methods Appl. Sci.}, 25(3):521--564, 2015.

\bibitem{ColomboGaravelloLecureux2012}
R.~M. Colombo, M.~Garavello, and M.~L{\'e}cureux-Mercier.
\newblock A class of nonlocal models for pedestrian traffic.
\newblock {\em Math. Models Methods Appl. Sci.}, 22(4):1150023, 34, 2012.

\bibitem{ColomboHertyMercier}
R.~M. Colombo, M.~Herty, and M.~Mercier.
\newblock Control of the continuity equation with a non local flow.
\newblock {\em ESAIM Control Optim. Calc. Var.}, 17(2):353--379, 2011.

\bibitem{ColomboMercier}
R.~M. Colombo and M.~Mercier.
\newblock An analytical framework to describe the interactions between
  individuals and a continuum.
\newblock {\em Journal of Nonlinear Science}, 22(1):39--61, 2012.

\bibitem{ColomboPogodaev2012}
R.~M. Colombo and N.~Pogodaev.
\newblock Confinement strategies in a model for the interaction between
  individuals and a continuum.
\newblock {\em SIAM J. Appl. Dyn. Syst.}, 11(2):741--770, 2012.

\bibitem{EvansGariepy}
L.~C. Evans and R.~F. Gariepy.
\newblock {\em Measure theory and fine properties of functions}.
\newblock Studies in Advanced Mathematics. CRC Press, Boca Raton, FL, 1992.

\bibitem{MR3431278}
R.~Hegselmann and U.~Krause.
\newblock Opinion dynamics under the influence of radical groups, charismatic
  leaders, and other constant signals: a simple unifying model.
\newblock {\em Netw. Heterog. Media}, 10(3):477--509, 2015.

\bibitem{kruzkov}
S.~N. Kru{\v{z}}hkov.
\newblock First order quasilinear equations with several independent variables.
\newblock {\em Mat. Sb. (N.S.)}, 81 (123):228--255, 1970.

\bibitem{LeVequeBook2002}
R.~J. LeVeque.
\newblock {\em Finite volume methods for hyperbolic problems}.
\newblock Cambridge Texts in Applied Mathematics. Cambridge University Press,
  Cambridge, 2002.

\bibitem{4118472}
R.~Olfati-Saber, J.~A. Fax, and R.~M. Murray.
\newblock Consensus and cooperation in networked multi-agent systems.
\newblock {\em Proceedings of the IEEE}, 95(1):215--233, Jan 2007.

\bibitem{PiccoliPouradierScharf}
B.~Piccoli, N.~Pouradier~Duteil, and B.~Scharf.
\newblock Optimal control of a collective migration model.
\newblock {\em Math. Models Methods Appl. Sci.}, 26(2):383--417, 2016.

\bibitem{MR3432849}
B.~Piccoli, F.~Rossi, and E.~Tr\'elat.
\newblock Control to flocking of the kinetic {C}ucker-{S}male model.
\newblock {\em SIAM J. Math. Anal.}, 47(6):4685--4719, 2015.

\bibitem{Schwartz}
J.~T. Schwartz.
\newblock {\em Nonlinear functional analysis}.
\newblock Gordon and Breach Science Publishers, New York-London-Paris, 1969.
\newblock Notes by H. Fattorini, R. Nirenberg and H. Porta, with an additional
  chapter by Hermann Karcher, Notes on Mathematics and its Applications.

\bibitem{MR3285316}
S.~Wongkaew, M.~Caponigro, and A.~Borz\`\i.
\newblock On the control through leadership of the {H}egselmann-{K}rause
  opinion formation model.
\newblock {\em Math. Models Methods Appl. Sci.}, 25(3):565--585, 2015.

\end{thebibliography}

}

\end{document}